\documentclass[11pt]{article}
\usepackage[numbers, sort]{natbib}
\usepackage{amsmath}
\usepackage{amsfonts}
\usepackage{amssymb}
\usepackage{epsfig}
\usepackage{amscd}
\usepackage{latexsym}
\usepackage[colorlinks=black,linkcolor=red,anchorcolor=blue,citecolor=green]{hyperref}
\usepackage{CJK}
\usepackage{color}
\usepackage{graphicx}
\usepackage{subfigure}

\newtheorem{theorem}{\bf Theorem}[section]

\newtheorem{lemma}[theorem]{\bf Lemma}
\newtheorem{proposition}[theorem]{\bf Proposition}
\newtheorem{corollary}[theorem]{\bf Corollary}

\newtheorem{definition}[theorem]{\bf Definition}

\newenvironment{proof}{\noindent{\em Proof:}}{\quad \hfill$\Box$\vspace{2ex}}

\newcommand\bcdot{\ensuremath{%
  \mathchoice%
   {\mskip\thinmuskip\lower0.2ex\hbox{\scalebox{1.5}{$\cdot$}}\mskip\thinmuskip}}%
   {\mskip\thinmuskip\lower0.2ex\hbox{\scalebox{1.5}{$\cdot$}}\mskip\thinmuskip}%
   {\lower0.3ex\hbox{\scalebox{1.2}{$\cdot$}}}%
   {\lower0.3ex\hbox{\scalebox{1.2}{$\cdot$}}}%
}

\def \bC{\mathbb{C}}
\def \bO{\mathcal{O}}

\def \bN {\mathbb N}
\def \bZ {\mathbb Z}

\def \aA {\mathbb A}
\def \aU {\mathbb U}
\def \aV {\mathbb V}

\def \bK {\mathcal{K}}
\def \bP {\mathcal{P}}
\def \bA {\mathcal{A}}

\def \bT {\mathcal{T}}
\def \bL {\mathcal{L}}
\def \bI {\mathcal{I}}
\def \bC {\mathcal{C}}

\def \BE {\textbf{E}}
\def \BK {\textbf{K}}
\def \BA {\textbf{A}}
\def \Bt {\textbf{t}}
\def \Bb {\textbf{b}}
\def \BG {\textbf{G}}
\def \Bf {\textbf{f}}
\def \Ba {\textbf{a}}

\pdfoutput=1
\begin{document}
\begin{CJK}{GBK}{song}

\title{\bf Oscillation Preserving Galerkin Methods for Fredholm Integral Equations of the Second Kind with Oscillatory Kernels}

\author{
    Yinkun Wang
    \footnotemark[2]
    \footnotemark[4] \
and
    Yuesheng Xu 
    \footnotemark[3]
    \footnotemark[4]}

\renewcommand{\thefootnote}{\fnsymbol{footnote}}

\footnotetext[2]{College of Science, National University of Defense Technology, Changsha 410073, People's Republic of China. Email: yinkun5522@163.com.}
\footnotetext[3]{Guangdong Province Key Lab of Computational Science, School of Mathematics and Computational Science, Sun Yat-Sen University, Guangzhou 510275, People's Republic of China. All correspondence should be sent to this author.}
\footnotetext[4]{Department of Mathematics, Syracuse University, Syracuse, NY 13244, USA. Email: yxu06@syr.edu.}

\date{January 18, 2015}
\maketitle

\begin{abstract}
{
Solutions of Fredholm integral equations of the second kind with oscillatory kernels likely exhibit oscillation. Standard numerical methods applied to solving equations of this type have poor numerical performance due to the influence of the highly rapid oscillation in the solutions. Understanding of the oscillation of the solutions is still inadequate in the literature and thus it requires further investigation. For this purpose, we introduce a notion to describe the degree of oscillation of an oscillatory function based on the dependence of its norm in a certain function space on the wavenumber. Based on this new notion, we construct structured oscillatory spaces with oscillatory structures. The structured spaces with a specific oscillatory structure can capture the oscillatory components of the solutions of Fredholm integral equations with oscillatory kernels. We then further propose oscillation preserving Galerkin methods for solving the equations by incorporating the standard approximation subspace of spline functions with a finite number of oscillatory functions which capture the oscillation of the exact solutions of the integral equations. We prove that the proposed methods have the optimal convergence order uniformly with respect to the wavenumber and they are numerically stable. A numerical example is presented to confirm the theoretical estimates.

}
\end{abstract}



\section{Introduction}
\label{sec;1}

We consider in this paper numerical solutions of Fredholm integral equations of the second kind with a highly oscillatory kernel. Specifically, the kernel is a product of a non-oscillatory smooth function and a typical known oscillatory function. We assume here that the forcing function is highly oscillatory with its oscillation being a combination of harmonic waves. Such an assumption is motivated from the plane waves which are used as incident waves in electromagnetic scattering problems. The solutions of the integral equations may possess certain oscillation and the error of numerical solutions by standard numerical methods may be greatly affected when the oscillation behaves rapidly. Conventional methods fail to solve equations of this type. This demands a better understanding of the solutions and devise new numerical methods in solving them.

We choose this type of oscillatory integral equations to study here because they resemble certain equations of optics and acoustics \cite{URSELL1969} and on the other hand they are simple enough to be rigorously analyzed to have a good understanding on the oscillation of their solutions. This paper is a beginning effort on the understanding of oscillation of the solution of highly oscillatory integral equations and the analysis methods proposed here may supply a potential tool from the mathematical viewpoint to understand highly oscillatory phenomena in scientific and engineering applications such as in electromagnetics and acoustics \cite{COLTON2013} and in laser theory \cite{BRUNNER2010,BRUNNER20102,BRUNNER2011}.


We remark that several other authors have considered the same integral equation \cite{URSELL1969,BRUNNER2010}. They mainly presented the asymptotic properties of its solution as the wavenumber tends to the infinity. In \cite{URSELL1969}, the asymptotic behavior of the solution was studied with the help of two associated Volterra equations and it was proved that the maximum norm of the solution is bounded by the product of the maximum norm of the right hand side function of the equation and a constant independent of the oscillation of the kernel. A refined result on the asymptotic behavior was presented in \cite{BRUNNER2010} with the assumption that the non-oscillatory function in the kernel is smooth and the right hand side function is non-oscillatory and smooth. The analysis was conducted through the Neumann series associated with the integral equation. Though there is no reference on the numerical solution of this kind of integral equation, two possible ways were proposed for the case of the non-oscillatory function in the kernel being $1$ in \cite{BRUNNER2010}. One way is to solve the integral equations by using the eigenvalues of the integral operator and its corresponding eigenfunctions. The shortcoming of this method is that it needs extra cost to compute eigenvalues and eigenfunctions. The other way is to approximate directly by the truncated Neumann expansion. This method has the poor convergence when the oscillation of the kernel is moderate. Both the methods may be difficult to implement, especially when the non-oscillatory factor of the kernel becomes complicated.

We now review a hybrid numerical method for solving boundary integral equations of the boundary value problems of Helmholtz equations (for example, see \cite{CHANDLER2012} and references therein). The hybrid method combines the conventional piecewise polynomial approximation with high-frequency asymptotics to build the basis functions suitable to represent their oscillatory solutions. The idea was first proposed in \cite{THIELE1975} in study of the electromagnetic characterization of wire antennas on or near a three-dimensional metallic surface. An ansatz of the high-frequency asymptotics is the base for the hybrid methods and it may be derived through the high frequency physical optics (or Kirchhoff) approximation \cite{KIRCHHOFF}, more precisely by the geometrical theory of diffraction \cite{KELLER1995} or by combining the former approximation with microlocal analysis \cite{MELROSE1985}. The ansatz indicates that the solutions of Helmholtz equations may be represented in terms of the product of an explicit oscillatory function and a less oscillatory unknown amplitude in each zone. The solutions of Helmholtz equations can be obtained by approximating the amplitude with a conventional integral method such as the Nystr\"{o}m method \cite{BRUNO2004}, the collocation method \cite{GILADI2001} and the Galerkin method \cite{DOMINGUEZ2007,CHANDLER2004}. Another method in solving Helmholtz equations is the partition of unity in which a number of plane waves is introduced on each element in addition to standard piecewise polynomial boundary elements \cite{BOUR1994,DEBAIN2003}. This method does not require any priori knowledge of the asymptotics of the solution which may result in the loss of uniform accuracy with respect to the wavenumber. To maintain the effective error, the degrees of freedom need to increase in proportion to the power of the wavenumber as the wavenumber tends to the infinity, just as for conventional methods, albeit with a lower constant of proportionality.

There was a recent development in solving the oscillatory Volterra integral equation  \cite{BRUNNER2014,WANG2011,XIANG2011,XIANG2013}. A Filon-type method was proposed in \cite{WANG2011} for the numerical solution of the Volterra integral equation of the first kind with a highly oscillatory Bessel kernel, based on the fact that its solution has an explicit integral expression \cite{WANG2011}. A Clenshaw-Curtis-Filon-type method was developed for computing the highly oscillatory Bessel transform and was used in solving the oscillatory Volterra integral equation in \cite{XIANG2011}. A Filon-type method and two collocation methods were presented in \cite{XIANG2013} for weakly singular Volterra integral equations of the second kind with a highly oscillatory Bessel kernel. These methods were designed based on the asymptotic analysis of the solutions of the equations. More recently, it was studied in \cite{BRUNNER2014} the high-oscillation property of the solutions of the integral equations associated with two classes of Volterra integral operators: compact operators with highly oscillatory kernels that are either smooth or weakly singular and noncompact cordial Volterra integral operators with highly oscillatory kernels. However, numerical analysis of the Volterra integral equations with a highly oscillatory kernel remains a challenging problem.

Since solutions of weakly singular integral equations has inspired us in developing efficient numerical methods for solutions of integral equations with highly oscillatory kernels, it deserves to review some related work in the subject. In \cite{BRUNNER1983}, non-polynomial spline collocations for the Volterra integral equation of the second kind with a weakly singular kernel were proposed by making use of the fact that the singularity behavior of its solution can be captured by certain special non-polynomial functions. Singularity preserving projection methods were developed in \cite{CAO1994} for the Fredholm integral equation of the second kind with a weakly singular kernel. These methods gave an optimal order of convergence for the approximation solutions since the singularity preserving approximation space can preserve the singularity of the solutions. Hybrid collocation methods were proposed in \cite{CAO2003} and \cite{CAO2007}, respectively, for the Volterra and Fredholm integral equations with weakly singular kernels. The basic idea of these methods is to enrich the basis of a standard approximate space with specific singular functions which can accurately capture the singularity characteristics of the solutions.

In this paper, we make three main contributions to the literature of the solution of oscillatory integral equations.
We first introduce a new notion to measure the degree of oscillation for an oscillatory function and then construct  structured spaces with oscillatory structures. The notion of oscillation is defined based on the dependence of the norm of a function in a certain space on the wavenumber.
It reflects the effect of the oscillation of an oscillatory function on the accuracy of its approximation. Structured spaces with specific oscillatory structures can capture the oscillation of solutions of oscillatory integral equations.
 The element in such a space may be approximated by an appropriate finite dimensional approximation space with its error independent of the oscillation even though the element may be highly oscillatory.
We then explore the oscillatory property of the solutions of the oscillatory integral equations by using the iterated integral operators instead of the asymptotic property of the solutions. The solutions will be proved to be in a non-oscillatory structured space with a specific oscillatory structure. Since the solutions can be represented by the iterated integral operators, the oscillation of the solutions depends on two properties of the iterated integral operators: The non-oscillatory structured space is closed under the iterated integral operators and they can reduce the oscillatory degree of oscillatory functions in the sense of the new notion of oscillation.
Finally, we develop the oscillation preserving Galerkin methods (OPGM) to solve the oscillatory integral equations based on the understanding of the oscillation of their solutions. The introduction of the methods is inspired by the singularity-preserving methods developed for solving the singular Fredholm integral equations of the second \cite{CAO1994} and the hybrid methods used in solving the scattering problems. The proposed methods have the optimal convergence order uniform with respect to the wavenumber and they are numerically stable when the wavenumber is large enough.

This paper is organized in eight sections. In section 2, we analyze
the oscillatory property of the solution of the oscillatory integral
equation of the second kind and the effect of the possible
oscillation of the solution on the accuracy of the conventional
method in solving the equation. We propose in section 3 a notion to
characterize the oscillatory degree of an oscillatory function. We
then construct the corresponding oscillatory spaces and structured
spaces with oscillatory structures. We present in section 4 special
results of the structured oscillatory space associated with the
Sobolev space. Specifically, we show that it is closed under a set
of oscillatory Fredholm integral operators. In section 5, we show
that the solutions of the oscillatory integral equations belong to
the structured oscillatory space associated with the Sobolev space.
In section 6, we develop the OPGM for solving the equation and prove
the convergence order (uniform with respect to the wavenumber) of
the proposed method. We analyze in section 7 the stability of a
special OPGM which uses the B-spline basis on a uniform partition.
In Section 8, we discuss the computational implementation of the
OPGM. A numerical example is given to illustrate the numerical
efficiency and accuracy of the proposed method in comparison with
the conventional Galerkin method.

\section{Solutions of Oscillatory Fredholm Integral Equations}
\label{sec;2}

We study in this section the oscillatory property of the solution of the oscillatory Fredholm integral equation of the second kind.

We begin with describing the integral equation. Let $I:=[-1,1]$. By $C(I)$ we denote the space of continuous complex-valued functions on $I$ and $C(I^2)$ the space of continuous complex-valued bivariate functions on $I^2$.  Supposing that $K\in C(I^2)$ and $f\in C(I)$ we consider the oscillatory integral equation
\begin{equation}\label{R9E15}
    y(s)-\int_I K(s,t)e^{i\kappa |s-t|}y(t)dt=f(s), \ \ s\in I,
\end{equation}
where $\kappa$ is the parameter of wavenumber and $y$ denotes the solution to be determined. In highly oscillatory problems, $\kappa\gg 1$. In this paper, we assume that the range of $\kappa$ is $(1,\infty)$.
We also assume that $K$ is independent of $\kappa$ while $f$ may has some oscillation of wavenumber $\kappa$.
Defining the integral operator $\bK$ by
\begin{equation}\label{R18E10}
(\bK y)(s):=\int_I K(s,t)e^{i\kappa|s-t|}y(t)dt, \ \ s\in I,
\end{equation}
we may rewrite equation (\ref{R9E15}) in its operator form
\begin{equation}\label{R9E16}
    (\bI-\bK) y=f,
\end{equation}
where $\bI$ denotes the identity operator.
It is well-known that the integral operator $\bK$ is compact on $C(I)$ and if 1 is not an eigenvalue of $\bK$ for any $\kappa$, then equation (\ref{R9E16}) has a unique solution in $C(I)$.

We now study how the solution $y$ of equation (\ref{R9E16}) depends on the wavenumber $\kappa$. It was proved in \cite{URSELL1969} that there exist positive constants $c$ and $\kappa_0$ such that for all $\kappa\geq \kappa_0$
$$
\|y\|_\infty \leq c \|f\|_\infty.
$$
This estimate indicates that the oscillation of the solution $y$ depends completely on the oscillation of the right hand side function $f$ in the uniform norm. However, this result does not offer any information about the derivatives of the solution $y$.
For the purpose of efficiently solving equation (\ref{R9E16}), we are interested in understanding how the derivatives of its solution $y$ depend on the wavenumber $\kappa$.

Specifically, we shall bound the Sobolev norm of the solution $y$ above by the power of the wavenumber $\kappa$.
To this end, we let $L^2(I)$ denote the space of functions $u$ on the interval $I$ for which
$$
\|u\|_2:=\left(\int_I|u(t)|^2dt\right)^{1/2}< \infty.
$$
Without ambiguity, we also use the notation $\|\cdot\|_2$ to represent the $L^2$ norm of bivariate functions, that is, $$
\|L\|_2:=\left(\int_{I^2}|L(s,t)|^2dsdt\right)^{1/2}, \ \
\mbox{for}\ \ L\in L^2(I).
$$
We next review the notation of the Sobolev space. Let $\bN_0:=\bN\cup \{0\}$ and for $n\in\bN$ let $\bZ_{n}:=\{0,1,\ldots,n-1\}$. For $n\in\bN_0$, we let $H^n(I)$ denote the Sobolev space $\{u\in L^2(I):u^{(n)}\in L^2(I)\}$ with its norm defined by
$$
\|u\|_{H^n}:=\left(\sum_{j\in \bZ_{n+1}}\left\|u^{(j)}\right\|_2^2\right)^{1/2}.
$$
We denote by $H^n_\kappa(I)$ the space of $\kappa$-parameterized functions $u$ which satisfy that
$u\in H^n(I)$ for any $\kappa$ and
there exist positive constants $c$ and $\kappa_0$ such that for all $\kappa\geq\kappa_0$,
$$
\|u\|_{H^p}\leq c\kappa^p, \ \ \mbox{for all}\ \ p\in\bZ_{n+1}.
$$

Associated with the operator $\bK$,
we introduce two auxiliary operators from $L^2(I)$ to $C(I)$. Specifically, for $L_j\in C(I^2), j=1,2$, and $u\in L^2(I)$, we let
\begin{equation}\label{R18E1}
\left(\bA_\epsilon[L_1,L_2]u\right)(s):=\int_{-1}^sL_1(s,t)e^{i\kappa (s-t)}u(t)dt+(-1)^\epsilon \int_s^1L_2(s,t)e^{i\kappa (t-s)}u(t)dt, \ \ s\in I,
\end{equation}
with $\epsilon\in\{0,1\}$. If $L_1=L_2$, we write $\bA_\epsilon[L_1,L_2]$ as $\bA_\epsilon[L_1]$. When $L_1=L_2=K$, we have that $\bK=\bA_0[K]$. By the Cauchy-Schwarz inequality, we can obtain that
\begin{equation}\label{R11E15}
\|\bA_\epsilon[L_1,L_2]u\|_2\leq \sqrt{\left( \|L_1\|_2^2+\|L_2\|_2^2\right)}\|u\|_2, \;\epsilon\in\{0,1\}.
\end{equation}
We next present an auxiliary lemma on the derivatives of operators $\bA_\epsilon$. To this end, for $m,n\in\bN_0$ we let $C^{[m,n]}(I^2)$ denote the space of functions whose derivatives of order up to $m$ and $n$ with respect to the first and second variables, respectively, are continuous. When $m=n$, we write $C^{[m,n]}(I^2)$ as $C^{[m]}(I^2)$.
In the next lemma, we let $D:=\frac{d}{ds}$ and use $c$ as a generic constant whose value may change in its appearance.

\begin{lemma}\label{R11L7}
Let $n\in\bN_0$. If  $K\in C^{[n+1,0]}(I^2)$ is independent of $\kappa$ and $u\in H^n_\kappa(I)$, then there exist positive constants $c$ and $\kappa_0$ such that for all $\kappa\geq \kappa_0$,
\begin{equation}\label{R18E20}
\left\|D^l\left( \bA_\epsilon \left[ K^{(j,0)} \right]u \right) \right\|_2\leq c\kappa^l, \;j\in \bZ_{n+2-l},\; l\in \bZ_{n+2}, \epsilon=0,1.
\end{equation}
\end{lemma}
\begin{proof}
 We prove this lemma by induction on $l$. The case $l=0$ follows directly from (\ref{R11E15}). We assume that (\ref{R18E20}) holds for $l<n+1$ and consider the case $l+1$. By applying $D$ to $\bA_\epsilon[K]u$, for $\epsilon=0,1$, we have that
\begin{equation}\label{R11E16}
D(\bA_\epsilon[K]u)(s)=\left(\bA_\epsilon\left[ K^{(1,0)}\right]u\right)(z)+i\kappa( \bA_{\widetilde{\epsilon+1}}[K]u)(s)+ 2\epsilon K(s,s)u(s),
\end{equation}
where $\widetilde{\epsilon+1}:=\left[(\epsilon+1)\mod 2\right]$. Then by applying $D^l$ to (\ref{R11E16}) with $K$ replaced by $K^{(j,0)}$ with $j\in\bZ_{n+1-l}$, we find that
\begin{equation}\label{R11E18}
\begin{split}
D^{l+1}\left(\bA_\epsilon \left[K^{(j,0)} \right]u \right) (s)=  D^l\left( \bA_\epsilon\left[K^{(j+1,0)} \right] u\right) (s) +i\kappa D^l\left( \bA_{\widetilde{\epsilon+1}}\left[K^{(j,0)} \right] u\right) (s)
+2 \epsilon D^l\left[ K^{(j,0)}(s,s) u(s)\right].\end{split}
\end{equation}
Using (\ref{R11E18}) with the induction hypothesis, we observe that estimate (\ref{R18E20}) holds for $l+1$. Thus,
by the induction principle, estimate (\ref{R18E20}) holds in general.
\end{proof}

Now we are ready to present the bound of the derivatives of the solution by the wavenumber $\kappa$.

\begin{theorem}\label{R11P2}
For a positive integer $m$, suppose that $K\in C^{[m,0]}(I^2)$ is independent of $\kappa$ and  $f\in H^m_\kappa(I)$. If $y$ is the solution of equation (\ref{R9E16}), then $y\in H^m_\kappa(I)$.
\end{theorem}
\begin{proof} It suffices to prove that there exist positive constants $c$  and $\kappa_0$ such that for all $p\in\bZ_{m+1}$ and $\kappa\geq\kappa_0$,
\begin{equation}\label{R8E33}
\|y\|_{H^p}\leq c\kappa^p.
\end{equation}
We prove this by induction on $p$. When $p=0$, we recall an expression in \cite{URSELL1969} for the solution $y$
\begin{equation}\label{R9E17}
y(s)=f(s)+\int_Ie^{i\kappa|s-t|}\{\Gamma_+(s,t)+\Gamma_-(s,t)\}f(t)dt+\int_I R(s,t;\kappa)y(t)dt,
\end{equation}
where $\Gamma_\pm$ independent of $\kappa$ are bounded functions and $R(s,t;\kappa)$ tends to zero uniformly as $\kappa$ tends to the infinity. By applying the $L^2$ norm to (\ref{R9E17}), we have that
$$
(1-\|R\|_2)\|y\|_2\leq (1+\|\Gamma_+ +\Gamma_-\|_2)\|f\|_2.
$$
With the properties of $\Gamma_\pm$ and $R$,  there exist two positive constants $\kappa_0\geq 1$ and $c$ such that for all $\kappa\geq \kappa_0$
\begin{equation}\label{R11E17}
\|y\|_2\leq c \|f\|_2.
\end{equation}
The case $p=0 $ follows directly from (\ref{R11E17}).

We assume that (\ref{R8E33}) holds for $p<m$  and we consider the case $p+1$. It suffices to show that there exist positive constants $c$  and $\kappa_0$ such that for all $\kappa\geq\kappa_0$,  $\|y^{(p+1)}\|_2\leq c\kappa^{p+1}$.
Since $y^{(p+1)}=(\bK y)^{(p+1)}+f^{(p+1)}$ and $f\in H^m_\kappa(I)$, it remains to prove that the $L^2$ norm of $(\bK y)^{(p+1)}$ is bounded by $c\kappa^{p+1}$. Note that $K\in C^{[p+1,0]}(I^2)$. Setting $l=p+1$ and $j=0$ in (\ref{R18E20}), by Lemma \ref{R11L7}, we have that $$
\|D^{p+1}Ky\|_2=\|D^{p+1}\left(\bA_0[K]y\right)\|_2\leq c\kappa^{p+1}.
$$
Hence, (\ref{R8E33}) holds for the case $p+1$.
\end{proof}

From the inequality (\ref{R11E17}), it is clear that the inverse of $\bI-\bK$ is bounded by a constant independent of $\kappa$. We state this result as a corollary below.

\begin{corollary}\label{R11C2}
If $K\in C(I^2)$ is independent of $\kappa$ and 1 is not an eigenvalue of $\bK$ for any $\kappa$, then the inverse of $\bI-\bK$ exists for any $\kappa$ and  there exist positive constants $c$ and $\kappa_0$ such that
\begin{equation}\label{R18E18}
\sup_{\kappa\geq\kappa_0}\|(\bI-\bK)^{-1}\| \leq c.
\end{equation}
\end{corollary}

When the solution of equation (\ref{R9E16}) is highly oscillatory, conventional numerical methods may fail to solve the equation. We now elaborate this point by taking the Galerkin methods as an example. Let $S_h^m$ denote the space of spines of order $m$ with $h$ being the maximal distance between two successive knots. This space will be defined precisely in Section \ref{sec;6}. Let $y_h\in S_h^m$ be the approximate solution of equation (\ref{R9E16}) obtained by the Galerkin methods. According to \cite{SLOAN1991}, if the solution $y\in H^m(I)$, then there exists a constant $h_0>0$ such that for all $0<h<h_0$,
$$
\|y_h-y\|_2\leq ch^m\|y\|_{H^m}.
$$
By Theorem \ref{R11P2}, the $H^m$ norm of the solution of equation (\ref{R9E16}) may increase in order $\bO(\kappa^m)$. As a result, there exists a constant $c$ independent of $\kappa$ such that
\begin{equation*}
\|y_h-y\|_2\leq c(\kappa h)^m.
\end{equation*}
To ensure convergence of the approximate solution $y_h$, we must choose the step size $h$ so that $h\kappa<1$. When $\kappa$ is large, $h$ will be small. Thus, the resulting linear system will have a large dimension and it is computationally costly to solve such a system. This explains why conventional numerical methods may fail to the equation when $\kappa$ is large. This motivates us to develop efficient non-conventional numerical methods for solving the equation. The difficulty of a conventional numerical method comes from the rapid oscillation of the solution. To overcome the difficulty, we are required to understand the oscillatory property of the solution. For this purpose, we shall introduce in the next section functional spaces suitable for oscillatory functions.

\section{Spaces of Oscillatory Functions}
\label{sec;3}

The main purpose of this section is to introduce a notion which describes the oscillation of an oscillatory function. Specifically, we define two kinds of oscillatory spaces, $\kappa$-oscillatory spaces of order $n$ and $\kappa$-oscillatory structured spaces of order $n$ with oscillatory structures.

To our best knowledge, there is no appropriate space in the present literature to describe oscillatory functions in the context of this paper. Though the space of the functions of bounded mean oscillation (BMO) is relates to oscillation of functions, it can not be used to describe the oscillatory properties of oscillatory functions. In fact, it is mainly used to study singular operators. Our first task is to introduce appropriate spaces to study oscillatory functions. Then what is an oscillatory function? To have a better view on the definition of oscillation, we first review the traditional understanding of oscillation. The oscillation was first used in describing the solution behaviour of second order linear differential equations \cite{LEIGHTON1950} and then was extended to describe the solution behaviour of Volterra integral equations. A commonly used definition for a solution being oscillatory is given as follows.
\begin{definition}
A solution $x(t)$ is said to be oscillatory if $x(t)$ has zeros for arbitrary large $t$ (i.e. it has infinitely many zeros for $t\geq0$ ); otherwise, a solution $x(t)$ is said to be non-oscillatory \cite{HARTMAN1952,BHATIA1966,ONOSE1975,ONOSE1990}.
\end{definition}
Though the oscillation defined above is for the solution of some differential and integral equations, it is can be extended to the definition of some more general oscillatory functions.
 \begin{definition}\label{oscillation-definition}
   A function $y(t)$ is said to be oscillatory if there exists a known non-oscillatory function $x(t)$ such that $y(t)-x(t)$ has zeros for arbitrary large $t$ (i.e. it has infinitely many zeros for $t\geq0$ ); otherwise, a solution $y(t)$ is said to be non-oscillatory.
\end{definition}
We note that Definition \ref{oscillation-definition} need our intuitive observation of oscillating phenomena to determine a non-oscillatory function. According to the extended definition of oscillation, it is defined on an infinite domain. Thus it can not give any information of oscillation such as the degree of oscillatory extent on finite domain of interest. The influence is that it may be helpless in numerical analysis in the finite domain. A simple example is that $\sin x$ and $\sin 10^6 x$ on $I:=[0,1]$ which are both oscillatory according to Definition \ref{oscillation-definition}.
The first function can be approximated numerically using much less computation than the second one by standard methods to obtain the same accuracy. Definition \ref{oscillation-definition} offers little knowledge on this difference.
The main problem here is related to another question which is ``how oscillatory is an oscillatory function?''.

More recently, high oscillatory problems have been extensively studied by a group of researchers at the Isaac Newton Institute of Mathematical Sciences. However, there is no a clearly stated mathematical definition on oscillation but it reveals that a function is highly oscillatory if it has rapidly oscillating phenomena \cite{ISERLES2009book}. It is a conclusion of oscillating phenomena and we can tell which function is more oscillating from the oscillating speed, called frequency in time domain or wavenumber in space domain. However, the understanding of this oscillation with frequency may still be helpless for numerical analysis purpose. It is because in the functional approximation context, what really matter is not just how quickly a function oscillates but the effect of oscillation of the function on the accuracy of its approximation. The wavenumber alone is not sufficient to describe the {\it effect} of oscillation of an oscillatory function that has on the approximation accuracy of the function.
%
%
We now illustrate this point by an example. We consider the functions
$$
g_j(t)=t^2+\frac{\sin(\kappa t)} {\kappa^{j-1}},\quad t\in I,\; j\in\bZ_4^+,
$$
where $\bZ_n^+:=\bZ_n\setminus \{0\}$, for $2\leq n\in\bN$. Clearly, functions $g_j$, $j\in\bZ_4^+,$  have the same wavenumber $\kappa$ and they oscillate rapidly in the same degree when $\kappa$ is large. However, the effect of
oscillation of these functions on the accuracy of their approximation are not the same. To see this point, we approximate them by using piecewise linear polynomial interpolations on the uniform partition of $I$ with 1281 points. We list in Table \ref{R11Ta5} the maximum errors of the approximation of these functions, where the norm  $\|\cdot\|_\infty$ is computed by sampling uniformly over $I$ with 2049 points.

\begin{table}[htb]
\begin{center}
\small
\caption{The maximum errors of the linear interpolations of $g_j$ on a uniform partition with 1281 points over $I$. }
\label{R11Ta5}
\begin{tabular}{cccc}
\hline
$\kappa$ & $\|g_1-\bar{g}_1\|_\infty$  & $\|g_2-\bar{g}_2\|_\infty$  & $\|g_3-\bar{g}_3\|_\infty$  \\
\hline
40&$4.89e-4$&$1.28e-5$&$9.16e-7$\\
80&$1.95e-3$&$2.50e-5$&$9.15e-7$\\
160&$7.80e-3$&$4.94e-5$&$9.15e-7$\\
320&$3.11e-2$&$9.78e-5$&$9.14e-7$\\
640&$1.22e-1$&$1.92e-4$&$9.09e-7$\\
\hline
\end{tabular}
\end{center}
\end{table}

The numerical results indicate that when the wavenumber $\kappa$ doubles, the approximation errors of $g_1$ increase about four times, those of $g_2$ are about doubled and those of $g_3$ show nearly no change.
This phenomenon is easily understood when we consider the approximation errors of these functions.
For each $j\in\bZ_4^+$, we let $\bar{g}_j$ denote the piecewise interpolation approximation of $g_j$. Then, it can be estimated that
\begin{equation}\label{concept_example}
\|g_j-\bar{g}_j\|_\infty\leq ch^2\|g_j^{(2)}\|_\infty\leq ch^2(2+\kappa^{3-j}), j\in\bZ_4^+,
\end{equation}
where $c$ is a constant independent of $\kappa$ and $h:=1/1280$. It means that for each $j\in\bZ_4^+$, the error of the approximation for $g_j$ increases at the speed of $\bO(\kappa^{3-j})$ as $\kappa\rightarrow \infty$. Estimate \eqref{concept_example} indicates that in the context of functional approximation, what really matters for the approximation accuracy is how the derivative (of the function to be approximated) which bounds the approximation error depends on the wavenumber.
Specifically, for $u\in C^m(I)$, we let
$$
\|u\|_{C^m}:=\max_{j\in\bZ_{m+1}}\|u^{(j)}\|_\infty
$$
and observe that in general the dependence of $\|u\|_{C^m}$ on the wavenumber $\kappa$ is crucial for the approximation of $u$ by piecewise polynomials of order $m$.

We now conclude that the traditional understanding of oscillation is a direct description of oscillating phenomena and there is no a suitable notion of oscillation which is useful in numerical analysis.  This motivates us to propose such a notion.
The observation made in the above example inspires us to introduce the new notion to describe the degree of oscillation of oscillatory functions based on the dependence of the norm in a certain function space on the wavenumber.

\begin{definition}\label{oscillatory-of-order-n}
Let $n$ be a positive number and $(X,\|\cdot\|_X)$ be a normed space.  A function $u$ is called $\kappa$-oscillatory of order $n$ in $X$ if it satisfies

(1) $u$ is $\kappa$-oscillatory in $X$, that is $u$ is $\kappa$-parameterized and $u\in X$ for any $\kappa$,

(2) there exist positive constants $c$ and $\kappa_0$ such that for all $\kappa\geq\kappa_0$
\begin{equation*}
\kappa^{-n}\|u\|_X\leq c.
\end{equation*}

When $n=0$, we say that $u$ is non-$\kappa$-oscillatory in $X$.
\end{definition}

In the functional approximation context, the space $X$ in Definition \ref{oscillatory-of-order-n} is normally a Sobolev space which appears in the error bound of an approximation of an oscillatory function. Its concrete form depends upon the regularity of $u$, the specific approximation space and the approximation principle are chosen. For example, if the function $u$ is $m$-times differentiable, the approximation space is chosen to be the splines of order $m$ and the approximation principle is the orthogonal projection (resp. interpolation), then the appropriate Sobolev space is $H^m$ (resp. $W^{\infty, m}$).

A simple example of $\kappa$-oscillatory function of order $2$ in $H^2(I)$ is $t^{5/2}e^{i\kappa t}$.
Based on the definition of $\kappa$-oscillatory, it is clear that the "oscillation" appeared in the function is only related to the parameter $\kappa$. We can also understand the $\kappa$-oscillatory function in some sense by the traditional oscillatory function with the wavenumber $\kappa$. Hence, the parameter $\kappa$ that appears may characterize how rapidly an oscillatory function oscillates.
 In Definition \ref{oscillatory-of-order-n}, the order $n$ reflects the degree of maximum influence that the oscillation of a function in space $X$ has on the accuracy of its approximation. The parameter $n$ can also be regarded as an index of the amplitude of the oscillatory function in the form of the norm of $X$. Though the concept is proposed to describe the oscillatory functions, some traditional oscillatory functions which are independent of $\kappa$ are non-$\kappa$-oscillatory and some $\kappa$-oscillatory functions can be non-oscillatory in traditional sense according to Definition \ref{oscillatory-of-order-n}. For example, a non-oscillatory function in traditional sense multiplied by a factor $\kappa^n$ may be a $\kappa$-oscillatory function of order $n$ in $C^n$ but not oscillatory according to the traditional definition.

The concept of $\kappa$-oscillatory functions of order $n$ is a natural development of the classification of oscillatory functions in some sence. In the past, we only distinguish the oscillatory functions by different wavenumbers or frequencies and now we enrich the classification by distinguishing the difference among traditional oscillatory functions of the same wavenumber from the point of their effect on the accuracy of the approximation.

%

Next, we shall define a $\kappa$-oscillatory space of order $n$ which gathers all the functions with the same maximum influence of the oscillation.
\begin{definition}\label{space-definition1}
Let $n$ be a positive number and $(X,\|\cdot\|_X)$ be a normed space. The set $X_{\kappa,n}:=\{u: u$ is $\kappa$-oscillatory of order $n$ in $X\}$ is called the $\kappa$-oscillatory space of order $n$.
The space $X_{\kappa, 0}$ is called a non-$\kappa$-oscillatory space.
\end{definition}

With Definition \ref{space-definition1} and the definition of space $H^m_\kappa(I)$ in Section \ref{sec;2}, we have that $H^m_\kappa(I)=\cap_{n=0}^m H^n_{\kappa,n}(I)$. Thus if the hypothesis of Theorem \ref{R11P2} are satisfied then the solution $y\in X_{\kappa, n}$ where $X:=H^n$, for $n\in\bZ_{m+1}$.

%

In the following, we shall introduce a kind of important space, $\kappa$--oscillatory structured space of order $n$ with an oscillatory structure.

\begin{definition}\label{space-definition2}
Let $n$ be a positive number, $(X,\|\cdot\|_X)$ be a normed space and $M\in \bN$. Suppose that $\varepsilon_M:=\left\{e_j: j\in\bZ_{M+1}^+\right\}$ is a set of typical known $\kappa$-oscillatory functions but also oscillatory functions in the traditional sense. A $\kappa$-oscillatory structured space of order $n$ with $\varepsilon_M$ is defined by
\begin{equation}\label{R18E38}
\tilde{X}_{\kappa,n}:=\left\{ u_0+\sum_{j=1}^M u_j e_j: u_j\in X_{\kappa,n}, j\in\bZ_{M+1}\right\},
\end{equation}
Space $\tilde{X}_{\kappa, 0}$ is called non-$\kappa$-oscillatory structured space with the oscillatory structure $\varepsilon_M$.
\end{definition}

It can be verified that for each $n$, space $\tilde{X}_{\kappa, n}$
is a vector space. Every
element in the space can be represented by the oscillatory
structure. For this reason, we call the space the
$\kappa$-oscillatory structured space. The choice of the oscillatory
structure $\varepsilon_M$ depends on the problem under consideration
and they should describe the main oscillation of interest. For
example, when solving equation (\ref{R9E15}), one may choose
$\varepsilon_M$ as $\{e^{i\kappa t}, e^{-i\kappa t}: t\in I\}$.
%
%

The non-$\kappa$-oscillatory structured space with a specific
structure is useful in approximating an oscillatory function. In
fact, if an oscillatory function is in $\tilde{X}_{\kappa,0}$, then
we are able to design numerical methods to approximate the
oscillatory function with the optimal convergence order as we do for
a non-oscillatory function. This is because we can construct a
corresponding structured approximation space with the same structure
to approximate functions in $\tilde{X}_{\kappa,n}$ to avoid the
effect from the high oscillation of the oscillatory structure. The
non-$\kappa$-oscillatory structured space with a specific structure
will serve as a main tool in describing the oscillatory properties
of the solution of equation (\ref{R9E15}).

%

\section{Oscillatory Function Spaces Associated with Sobolev Spaces}
\label{sec;4}

We present in this section special results of oscillatory function
spaces which defined in Section \ref{sec;3} associated with Sobolev
spaces. These spaces will be used to characterize in the next
section the solution of the oscillatory Fredholm integral equations
of the second kind.

Through out this paper, let $m$ be a fixed positive integer.
According to Definitions \ref{space-definition1} and
\ref{space-definition2}, associated with the Sobolev space $H^m(I)$,
for each $n\in\bZ_{m+1}$, we construct the $\kappa$-oscillatory
space $H_{\kappa,n}^m(I)$ of order $n$ and the corresponding
$\kappa$-oscillatory structured space $\tilde{H}^m_{\kappa,n}(I)$ of
order $n$ with the structure $\{e^{i \kappa t}, e^{-i \kappa t}:
t\in I\}$. In particular, $\tilde{H}^m_{\kappa,0}(I)$ is the
corresponding non-$\kappa$-oscillatory structured space. It will be
proved that $\tilde{H}^m_{\kappa,0}(I)$ is closed under a set of
Fredholm integral operators.

We begin with an investigation of an oscillatory Volterra integral
function related to operator $\bK$. We define the oscillatory
Volterra integral function below. For $a\in I$ and $L\in
C^{[n]}(I^2)$, we define the oscillatory Volterra integral function
by
$$
\ell(s):=\int_a^sL(s,t)e^{i\kappa t}dt, \ \ s\in I.
$$
We recall a well-known result  that if $u\in C^{n}[a,b], n\in \bN$,
then
\begin{equation}\label{R8E37}
    \int_a^bu(t)e^{i\kappa t}dt=e^{i\kappa b}\sigma_n[u](b)- e^{i\kappa a}\sigma_n[u](a)+\frac{(-1)^n}{(i\kappa )^{n}} \int_a^bu^{(n)}(t)e^{i\kappa t}dt,
\end{equation}
where
\begin{equation}\label{R8E1}
\sigma_n[u](t):=\sum_{j=0}^{n-1}\frac{(-1)^{j}}{(i\kappa )^{j+1}}u^{(j)}(t),\ \ t\in [a,b].
\end{equation}
The equation (\ref{R8E37}) can be easily obtained by employing
integral by parts. A more general form of (\ref{R8E37}) can be found
in \cite{NORSETT2005}. Motivated by (\ref{R8E37}), we define
\begin{equation}\label{R10E2}
  r(s) := -e^{i\kappa a} \sigma_{n} [L(s,\bcdot)](a),\ \ s\in I,
\end{equation}and
\begin{equation}\label{R10E1}
  g(s) :=\left[\ell(s)-r(s)\right]e^{-i\kappa s},\ \ s\in I.
\end{equation}
Thus, the oscillatory Volterra integral function $\ell$ has the
decomposition
\begin{equation}\label{R8E32}
\ell(s)=g(s)e^{i\kappa s}+r(s), \ \ s\in I.
\end{equation}
The expression of $r$ is clear and the property of its derivatives
can be easily obtained from those of the kernel $L$. However, the
property of the derivatives of $g$ requires further investigation.
We next present an explicit expression for the derivatives of $g$.
To this end, we first show an auxiliary equality. For
$j\in\bZ_{n+1}$ and $n\in \bN_0$, we let
$C_n^j:=\frac{n!}{j!(n-j)!}$. We also need the formula
\begin{equation}\label{R18E50}
[L(s,s)]^{(p)}=\sum_{b=0}^pC_p^b L^{(b,p-b)}(s,s),\ \ s\in I,
\end{equation}
where by $[L(s,s)]^{(p)}$ we mean $\frac{d^p}{ds^p}L(s,s)$.

\begin{lemma}\label{R11L8} If for $l\in \bN$, $L\in C^{[l-1]}(I^2)$, then for all $s\in I$
\begin{equation}\label{R18E39}
\sum_{p=0}^l\sum_{q=0}^{p-1}C_l^p(-i\kappa)^{l-p}e^{-i\kappa
s}\left[L^{(q,0)}(s,s)e^{i\kappa s}\right]^{(p-1-q)}=
\sum_{p=0}^{l-1}
\sum_{q=0}^{l-1-p}C_l^p(-i\kappa)^{l-1-p-q}L^{(q,p)}(s,s).
\end{equation}
\end{lemma}
\begin{proof}
 Using the Leibniz rule for high order derivatives, we get that
\begin{equation}\label{R18E40}
\begin{split}
e^{-i\kappa s}\left[L^{(q,0)}(s,s)e^{i\kappa s}\right]^{(p-1-q)}=
\sum_{\alpha=0}^{p-1-q} C_{p-1-q}^\alpha (i\kappa )^{p-1-q-\alpha}
\left[L^{(q,0)}(s,s)\right]^{(\alpha)}
\end{split},\ \ s\in I.
\end{equation}
We denote by $\Pi(s)$ the left hand side of the formula
(\ref{R18E39}). Substituting (\ref{R18E40}) into  $\Pi$ and then
making a change of the summation order of the resulting expression,
we have that
\begin{equation}\label{R18E21}
\begin{split}
\Pi(s)= \sum_{q=0}^{l-1}\sum_{\alpha=0}^{
l-1-q}\left\{\sum_{p=q+\alpha+1}^l
C_l^pC_{p-1-q}^\alpha(-1)^{p-1-q-\alpha }\right\}(-i\kappa
)^{l-1-q-\alpha} \left[L^{(q,0)}(s,s) \right]^{(\alpha)}
\end{split},\ \ s\in I.
\end{equation}
By the recursion relation
\begin{equation}\label{R8E34}
C_n^j=C_{n-1}^{j-1}+C_{n-1}^j, j\in\bZ_{n-1}^+, \ \ n\in \bN,
\end{equation}
 we can rewrite the expression in the curly braces of (\ref{R18E21}) as $C_{l-1-\alpha}^q$. Hence, (\ref{R18E21}) becomes
\begin{equation}\label{R18E23}
\begin{split}
\Pi(s)=\sum_{q=0}^{l-1}\sum_{\alpha=0}^{ l-1-q}C_{l-1-\alpha}^q (-i\kappa )^{l-1-q-\alpha} \left[L^{(q,0)}(s,s)\right]^{(\alpha)}
\end{split},\ \ s\in I.
\end{equation}
Equation (\ref{R18E23}) can be expanded further with (\ref{R18E50}).
Then we change the summation order of $q$ and of $\alpha$ and change
$q+\alpha$ to $q$. This leads to
\begin{equation}\label{R18E24}
\begin{split}
\Pi(s)= \sum_{\alpha=0}^{l-1}\sum_{q=\alpha}^{ l-1}\sum_{p=0}^\alpha C_{l-1-\alpha}^{q-\alpha} C_\alpha^p (-i\kappa )^{l-1-q} L^{(q-p,p)}(s,s) \end{split},\ \ s\in I.
\end{equation}
Changing the summation order again in (\ref{R18E24}) gives that
\begin{equation}\label{R18E25}
\begin{split}
\Pi(s)= \sum_{q=0}^{l-1}\sum_{p=0}^{q}\left(\sum_{\alpha=p}^{q}C_{l-1-\alpha}^{q-\alpha} C_\alpha^p\right) (-i\kappa )^{l-1-q} L^{(q-p,p)}(s,s),\ \ s\in I.
\end{split}
\end{equation}
Again, by applying the recursion relation (\ref{R8E34}), we get that
$\sum_{\alpha=p}^{q}C_{l-1-\alpha}^{q-\alpha} C_\alpha^p =C_l^{q-p}$
and (\ref{R18E25}) becomes
\begin{equation}\label{R18E26}
\begin{split}
\Pi(s)=\sum_{q=0}^{l-1}\sum_{p=0}^{q} C_l^{q-p} (-i\kappa )^{l-1-q}L^{(q-p,p)}(s,s),\ \ s\in I.
\end{split}
\end{equation}
Finally, we change the summation order of $p$ and of $q$ in
(\ref{R18E26}) and then  in the resulting formula we change $q-p$ to
$q$. This yields the desired formula (\ref{R18E39}).
\end{proof}

We are now ready to present the explicit expression for derivatives
of $g$. For $n,l\in \bN_0$, let $\aU_{n,l}:=\{(q,p):q\in\bZ_n,
p\in\bZ_{l+1}, q+p\geq l\}$.

\begin{lemma}\label{R11L1}
 If $L\in C^{[n]}(I^2), n\in\bN_0$, then for $l\in\bZ_{n+1}$
\begin{equation}\label{R18E47}
 g^{(l)}(s)=-\sum_{(q,p)\in\aU_{n,l}}C_l^p(-ik)^{l-1-p-q}L^{(p,q)}(s,s)+\sum_{p=0}^l C_l^p(-ik)^{l-p-n}\int_a^sL^{(p,n)}(s,t)e^{i \kappa(t-s)}dt.
\end{equation}
\end{lemma}
\begin{proof}
By the Leibniz formula we have that
\begin{equation}\label{R11E9}
\begin{split}
g^{(l)}(s)= \sum_{p=0}^lC_l^p(-i\kappa )^{l-p}e^{-i\kappa s} \left\{
\ell^{(p)}(s)-r^{(p)}(s) \right\},\ \ s\in I.
\end{split}
\end{equation}
A direct computation gives
\begin{equation}\label{R11E10}
\begin{split}
\ell^{(p)}(s)=& \sum_{q=0}^{p-1}  \left[L^{(q,0)}(s,s) e^{i\kappa s}\right]^{(p-1-q)} +\int_{a}^s L^{(p,0)}(s,t) e^{i\kappa t}dt.
\end{split}
\end{equation}
By applying (\ref{R8E37}) to the second term on the right hand side
of (\ref{R11E10}) and combining  definition (\ref{R10E2}) of $r$, we
obtain that
\begin{equation}\label{R11E11}
\begin{split}
\int_{a}^s L^{(p,0)}(s,t)  e^{i\kappa t}dt-r^{(p)}(s)=&e^{i\kappa
s}\sigma_{n} \left[L^{(p,0)}(s,\bcdot)
\right](s)+\frac{(-1)^n}{(i\kappa )^{n}} \int_{a}^s L^{(p,n)}(s,t)
e^{i\kappa t}dt.
\end{split}
\end{equation}
To simplify the expression of $g^{(l)}$, we substitute the
expression (\ref{R8E1}) of
$\sigma_{n}\left[L^{(p,0)}(s,\bcdot)\right](s)$ into (\ref{R11E11})
and then substitute (\ref{R11E10}) and (\ref{R11E11}) into
(\ref{R11E9}). Finally, by combining the simplified expression of
$g^{(l)}$ with Lemma \ref{R11L8} and using the notation $\aU_{n,l}$,
we obtain the desired formula (\ref{R18E47}) for $g^{(l)}$.
\end{proof}

We next show that $\tilde{H}^m_{\kappa,0}(I)$ is closed under a
Volterra integral operator to be defined below. For $a\in I$ and
$L\in C(I^2)$, we define the Volterra integral operator $\bC_a:
L^2(I)\rightarrow C(I)$ by
\begin{equation}\label{R18E3}
(\bC_a[L]u)(s):=\int_{a}^sL(s,t)e^{i\kappa|s-t|}u(t)dt,\ \ s\in I.
\end{equation}
For each $n \in\bN_0$, we define the norm for $C^{[n]}(I^2)$ by
$\|L\|_{C^{[n]}}:=\max_{p,q\in\bZ_{n+1}}\|L^{(p,q)}\|_\infty$.
Denote by $C^{[n]}_{\kappa,0}(I^2)$ the non-$\kappa$-oscillatory
space in $C^{[n]}(I^2)$. Thus, for $L\in C^{[n]}_{\kappa,0}(I^2)$,
there exists a positive constant $c$ independent of $\kappa$ such
that $\|L\|_{C^{[n]}}<c$.

\begin{lemma}\label{R10L2}
If $L\in C^{[m]}_{\kappa,0}(I^2)$, then $\tilde{H}^m_{\kappa,0}(I)$ is closed under $\bC_a[L]$.
\end{lemma}
\begin{proof}
In this proof, we shall write $\bC_a[L]$ as $\bC_a$ for convenience.
Note that for any $u\in \tilde{H}^m_{\kappa,0}(I)$, there exist
$v_j\in H_{\kappa,0}^m(I),j\in\bZ_4^+$ such that
$$
u(s)=v_1(s)+v_2(s)e^{i\kappa s} +v_3(s)e^{-i\kappa s}, \ \ s\in I.
$$
Hence, it suffices to prove that for each $v\in H_{\kappa,0}^m(I)$,
$\bC_av$, $\bC_a(ve^{i\kappa\bcdot})$, and $\bC_a(ve^{-i\kappa
\bcdot})$ are in $\tilde{H}^m_{\kappa,0}(I)$.

We first consider the case of $a<s$. In this case, by the definition
of $(\bC_av)$, we have that
$$
(\bC_av)(s)=e^{i\kappa s} \int_{a}^s L(s,t)v(t) e^{-i\kappa t}dt.
$$
To decompose $\bC_av$, we define two functions
\begin{equation*}
  {\tilde r}(s) := -e^{-i\kappa a} \sigma_{m} [L(s,\bcdot)v(\bcdot)](a),\ \ s\in I,
\end{equation*}and
\begin{equation*}
  {\tilde g}(s) :=\left[\int_{a}^s L(s,t)v(t) e^{-i\kappa t}dt-r(s)\right]e^{i\kappa s},\ \ s\in I.
\end{equation*}
Hence, we have that
\begin{equation}\label{R11E26}
 (\bC_a v)(s)= {\tilde g}(s)+{\tilde r}(s)e^{i\kappa s}, \ \ s\in I.
\end{equation}
It remains to prove that ${\tilde g}, {\tilde r}\in
H_{\kappa,0}^m(I)$. In fact, both ${\tilde g}$ and ${\tilde r}$
satisfy a somewhat stronger condition that they belong to
$C_{\kappa,0}^m(I)$, the non-$\kappa$-oscillatory space based on
$C^m(I).$ For each $l\in\bZ_{m+1}$, we apply $D^l$  to ${\tilde r}$
and take its maximum norm to obtain that
\begin{equation*}
\begin{split}
\left\|{\tilde r}^{(l)}\right\|_\infty\leq &
\sum_{q=0}^{m-1}\frac{1}{\kappa^{q+1}} \max_{s\in I}
\left|(Lv)^{(l,q)}(s,a) \right| \leq \frac{2^mm}{\kappa}\max_{q\in
\bZ_m}\|L^{(l,q)}\|_{\infty}\|v\|_{C^{m-1}}.
\end{split}
\end{equation*}
Since $v\in H_{\kappa,0}^m(I)$, there exists a positive constant $c$
independent of $\kappa$ such that  $\|v\|_{C^{m-1}} \leq
c\|v\|_{H^m}$. Thus, ${\tilde r}\in C_{\kappa,0}^m(I)$. By Lemma
\ref{R11L1}, we have that for $l\in\bZ_{m+1}$
\begin{equation}\label{tilde g}
{\tilde g}^{(l)}(s)
=-\sum_{(q,p)\in\aU_{m,l}}C_l^p(ik)^{l-1-p-q}(Lv)^{(p,q)}(s,s)
+\sum_{p=0}^l C_l^p(ik)^{l-p-m}\int_a^s(Lv)^{(p,m)}(s,t)e^{-i
\kappa(t-s)}dt.
\end{equation}
The first summation in the right hand side of \eqref{tilde g} is
bounded by
\begin{equation}\label{R18E30}
\begin{split}
\sum_{(q,p)\in \aU_{m,l}} C_l^{p}k^{l-1-p-q}\max_{s\in I}\left|
(Lv)^{(p,q)}(s,s)\right|\leq \frac{2^{m+l}m}{\kappa}\max_{(q,p)\in
\aU_{m,l}}\|L^{(p,q)}\|_\infty\|v\|_{C^{m-1}},\;
\end{split}
\end{equation}
and the second summation is bounded by
\begin{equation}\label{R18E31}
\begin{split}
\sum_{p=0}^lC_l^p\kappa^{l-p-m} \max_{s\in I} \int_{-1}^1\Bigg|
(Lv)^{(p,m)}(s,t)\Bigg|dt \leq
\sqrt{2}2^{2m+1}\kappa^{l-m}\max_{p\in\bZ_{l+1}}\|L^{(p,m)}\|_\infty\|v\|_{H^m},\
\ l\in\bZ_{m+1}.
\end{split}
\end{equation}
The combination of (\ref{R18E30}) and (\ref{R18E31}) yields a positive
constant $c$ independent of $\kappa$ such that
\[
\left\|{\tilde g}^{(l)}\right\|_\infty \leq
c\max_{p,q\in\bZ_{m+1}}\|L^{(p,q)}\|_\infty\|v\|_{H^m}.
\]
We conclude that ${\tilde g}\in  C_{\kappa,0}^m(I)$. Therefore,
$\bC_av\in \tilde{H}^m_{\kappa,0}(I)$. Likewise, we can show that
$\bC_a(ve^{i\kappa \bcdot}), \bC_a(ve^{-i\kappa \bcdot})\in
\tilde{H}^m_{\kappa,0}(I)$.

The case when $a\geq s$ can be similarly handled.
\end{proof}

We are now ready to prove the main result of this section. To this
end, we define the set $\aA$ of oscillatory Fredholm integral
operators by
\[
\aA:=\left\{\bT: \bT:=\bA_\epsilon[L_1,L_2],\epsilon=0,1,\; L_j\in
C^{[m]}_{\kappa,0}(I^2), j=1,2 \right\}.
\]
The next theorem reveals that the space $\tilde{H}^m_{\kappa,0}(I)$
is closed under summations of products of operators in $\aA$.

\begin{theorem}\label{R11L3}
If $\bT_j\in \aA$, $\beta_j\in\bN_0$, $\alpha_l\in\bN_0$ and
$n\in\bN$, space $\tilde{H}^m_{\kappa,0}(I)$ is closed under
$\sum_{l=1}^n\prod_{j=1}^{\alpha_l}\bT_j^{\beta_j}$.
\end{theorem}
\begin{proof}
The proof is based on the fact that a Fredholm integral operator can
be split into two oscillatory Volterra integral operators.
Specifically, an operator $\bT\in\aA$ can be split into two Volterra
integral operators of the same kind as $\bC_a$. Hence, according to
Lemma \ref{R10L2}, $\tilde{H}^m_{\kappa,0}(I)$ is closed under
$\bT$.

Suppose that $\bT$ is a product of any operators in $\aA$,  that is,
$\bT=\prod_{j=1}^n \bT_j$, $n\in \bN_0$ where $\bT_j\in \aA,
j\in\bZ_{n}^+$. Since for any $u\in \tilde{H}^m_{\kappa,0}(I)$
$\bT_j u\in \tilde{H}^m_{\kappa,0}(I)$, $j\in\bZ_{n}^+$, by
induction we observe that $\bT u\in \tilde{H}^m_{\kappa,0}(I)$ which
reveals that $\tilde{H}^m_{\kappa,0}(I)$ is closed under $\bT$.
\end{proof}

A direct application of Theorem \ref{R11L3} yields the following
corollary concerning the iterated operators of $\bK$.

\begin{corollary}\label{R11C1}
If $K\in C^{[m]}(I^2)$ is independent of $\kappa$, then
$\tilde{H}^m_{\kappa,0}(I)$ is closed under $\bK^n, n\in \bN$.
\end{corollary}
\begin{proof}
Since $\bK=\bA_0[K]$ and $K\in C^{[m]}(I^2)$ is independent of
$\kappa$, the result of this corollary follows directly from Theorem
\ref{R11L3}.
\end{proof}

Corollary \ref{R11C1} will play an important role in the next
section in understanding the oscillatory structure of the solution
of equation (\ref{R9E15}).

\section{Oscillation of Solutions of Oscillatory Fredholm Equations}
\label{sec;5}

We study in this section the oscillation property of the solution of
equation (\ref{R9E15}). This is done by considering the iterated
integral operators of $\bK$. Specifically, the main purpose of this
section is to establish that if $K\in C^{[m]}(I^2)$ is independent
of $\kappa$ and $f\in \tilde{H}^m_{\kappa,0}(I)$, the solution of
equation (\ref{R9E15}) belongs to $\tilde{H}^m_{\kappa,0}(I)$.

The solution $y$ of (\ref{R9E15}) can be represented by using
iterated integral operators of $\bK$. In fact, by successive
substitutions, we obtain from (\ref{R9E15}) that
\begin{equation}\label{R18E15}
y=\bK^ny+\sum_{j=0}^{n-1}\bK^jf,\ \ n\in\bN.
\end{equation}
It is clear that the oscillatory property of the solution depends on
the properties of the iterated integral operators of $\bK$. It
follows from Corollary \ref{R11C1} that if $f\in
\tilde{H}^m_{\kappa,0}(I)$, $\sum_{j=0}^{n-1}\bK^jf$ is in
$\tilde{H}^m_{\kappa,0}(I)$. It remains to prove that there exists a
number $n\in \bN$ such that $\bK^ny \in \tilde{H}^m_{\kappa,0}(I)$.
According to Theorem \ref{R11P2},
the solution of (\ref{R9E15}) belongs to $H^m_\kappa(I)(i.e. \cap_{n=0}^m H^n_{\kappa,n}(I))$. Hence, it
suffices to show that there exists an $n\in\bN$ such that $\bK^n:
H^m_\kappa(I)\rightarrow \tilde{H}^m_{\kappa,0}(I)$.

The next lemma shows that for $L\in C^{[m]}_{\kappa,0}(I^2)$
satisfying an additional condition, the integral operator $\bC_a[L]$
defined by (\ref{R18E3}) can ease oscillation of functions in
$H^m_\kappa(I)$.

\begin{lemma}\label{R11L12}
Let $n\in\bZ_{m+1}^+$. If $L\in C^{[m]}_{\kappa,0}(I^2)$  and  there
exists a positive constant $c_0$ independent of $\kappa$ such that
\begin{equation}\label{R8E36}
\max_{s\in I}\left| L^{(\alpha,\beta)}(s,s) \right|\leq
c_0\kappa^{-(n-\alpha-\beta-1)}, \ \ 0\leq \alpha+\beta+1\leq n,
\end{equation}
then $\bC_a[L]: H^m_\kappa(I)\rightarrow
\tilde{H}^{m}_{\kappa,m-n}(I)$.
\end{lemma}
\begin{proof} We provide proof for the case $a<s$ only since the other case can be
similarly handled. According to the definition (\ref{R18E3}), we
have that $\left(\bC_a[L] u\right)(s)=u_1(s)e^{i\kappa s}$, for
$s\in I,$ where $ u_1(s):= \int_{a}^s L(s,t)e^{-i\kappa t}u(t)dt$.
It suffices to prove that $u_1\in H^{m}_{\kappa,m-n}(I)$.

Differentiation of $u_1$ leads to
\begin{equation}\label{R10E9}
u_1^{(l)}(s)=\sum_{p=0}^{l-1}\sum_{q=0}^{l-1-p}C_{l-1-q}^p
\left[L^{(q,0)}(s,s)\right]^{(l-1-q-p)}\left[e^{-i\kappa s}u(s)
\right]^{(p)}+ \int_{a}^sL^{(l,0)}(s,t) e^{-i\kappa t}u(t)dt,\ \
l\in Z_{m+1}.
\end{equation}
By using the Cauchy-Schwarz inequality, the $L^2$ norm of the second
term on the right hand side of (\ref{R10E9}) is bounded by
$2\max_{q\in\bZ_{m+1}}\left\| L^{(q,0)}\right\|_{\infty}\|u\|_2$.

We next bound the first term (denoted by $\delta_l(z)$) of the right
hand side of (\ref{R10E9}). Since $u\in H^m_\kappa(I)$, by the
Leibniz formula, it is easy to verify that there exists a positive
constant $c$ independent of $\kappa$ such that
\begin{equation}\label{R18E12}
\left\|\left[e^{-ik\bcdot}u \right]^{(p)}\right\|_2\leq c\kappa^p,\ \ p\in \bZ_m.
\end{equation}
The $L^2$ norm of the terms with $p\leq l-n$ in $\delta_l$ is
bounded by $c\kappa^{l-n}$ according to (\ref{R18E12}) and $L\in
C^{[m]}_{\kappa,0}(I^2)$. If $p>l-n$, then the factor
$\left[L^{(q,0)}(s,s)\right]^{(l-1-q-p)}$ can be represented by a
linear combination of $L^{(\alpha,\beta)}(s,s)$,
$0\leq\alpha+\beta\leq l-1-p$. Thus, the $L^2$ norm of the terms in
$\delta_l(z)$ for $p>l-n$ is bounded by $c\kappa^{l-n}$ according to
the condition (\ref{R8E36}) and (\ref{R18E12}). Hence, there exists
a positive constant $c$ independent $\kappa$ such that
$\|\delta_l\|_2\leq c\kappa^{l-n}$, for all $l\in \bZ_{m+1}.$
Combining the bounds of the two terms in (\ref{R10E9}), we obtain
that there exists a positive constant $c$ independent $\kappa$ such
that $\|u_1^{(l)}\|_2\leq c\max\{\kappa^{l-n},1\}$, for all $l\in
\bZ_{m+1}$. Therefore, we have that $u_1\in H^m_{\kappa,m-n}(I)$.
\end{proof}

The oscillation order of a function in $H^m_\kappa(I)$ is $m$ and
that of a function in $\tilde{H}^m_{\kappa,m-n}(I)$ is $m-n$
according to the definition of $\kappa$-oscillatory function of
order $n$. In this sense, the Volterra integral operators $\bC_a[L]$
with $L\in C^{[m]}_{\kappa,0}(I^2)$ satisfying (\ref{R8E36}) can
reduce the oscillation order of functions in $H^m_\kappa(I)$ by $n$,
for $n\in\bZ_{m+1}^+$, according to Lemma \ref{R11L12}. We call a
bivariate function $L$ the less oscillatory kernel of order $n$ if
$L\in C^{[m]}_{\kappa,0}(I^2)$ and satisfies (\ref{R8E36}). In the
next lemma, we shall extend the result of Lemma \ref{R11L12} to the
Fredholm integral operator.

\begin{lemma}\label{R11L2}
Let $n\in\bZ_{m+1}^+$. If $L_j, j=1,2$, are the less oscillatory
kernels of order $n$, then for $\epsilon=0,1$,
$\bA_\epsilon[L_1,L_2]: H^m_\kappa(I) \rightarrow
\tilde{H}^m_{\kappa,m-n}(I)$.
\end{lemma}
\begin{proof}
Note that the operator $\bA_\epsilon[L_1,L_2]$ can be split as two
Volterra integral operators of the type $\bC_a[L]$. The result of
this lemma follows directly from Lemma \ref{R11L12}.
\end{proof}

We next study the effect of the iterated integral operators (IIO) of
$\bK$ applied to $H^m_\kappa(I)$. This is done by using induction on
the order of the iterated integral operators and finally choosing an
suitable $n\in \bN$ such that $\bK^nH^m_\kappa(I)\subset
\tilde{H}^m_{\kappa,0}(I)$. Notice that the iterated integral
operator of $\bK$ of order 1 is equal to $\bK$. Since $K\in
C^{[m]}(I^2)$ is independent of $\kappa$, we conclude that $K\in
C^{[m]}_{\kappa,0}(I^2)$ and there exists a constant $c$ independent
of $\kappa$ such that $\max_{z\in I}|K(z,z)|\leq c$. Hence, $K$ is
the less oscillatory kernel of order $1$. By Lemma \ref{R11L2},
$\bK$ maps $H^m_\kappa(I)$ into $\tilde{H}^m_{\kappa,m-1}(I)$. This
means that $\bK$ reduces the oscillatory order of functions in
$H^m_\kappa(I)$ by 1.

We shall show that for $n\in\bZ^+_{m+1}$, $\bK^n$ can reduce the
oscillatory order of functions in $H^m_{\kappa}(I)$ by $n$. Because
the IIO of a high order is the composition of $\bK$ and the IIO of
one order lower, the property of the IIO of a high order can be
derived through induction on the order. To decompose a composite
operator, we need to decompose its corresponding kernel. Since the
kernel of the composite operator is represented by an oscillatory
integral whose integrand has the non-differentiable oscillator, we
need to split the integral domain. To this end, we introduce
auxiliary functions to represent the decomposition of the kernel.
For $\epsilon=1,2$, we let $K_{1,\epsilon}:=K$. For $1<n\in \bN$ and
$\epsilon=1,2,$ we define
\begin{equation}
L_{n-1,\epsilon}(s,t,x):=K(s,x)K_{n-1,\epsilon}(x,t), \ \ (s,t,x)\in I^3.
\end{equation}
The function $L_{n-1,\epsilon}$ is the less oscillatory part in the
integrand of the kernel of  the composite operator. To represent the
decomposition,  for $1<n\in \bN$ and $\epsilon=1,2,$ we define
\begin{equation}
J_{n,\epsilon}(s,t):=(-1)^{\epsilon-1}\int_t^sL_{n-1,\epsilon}(s,t,x)dx, \ \ (s,t)\in I^2.
\end{equation}
Let $\aV:=\{(1,-1,2,t),(2,-1,2,s), (3,1,1,s), (4,1,1,t)\}.$ For $1<n\in \bN$ and $(j,\tilde{j},\epsilon_j, z) \in \aV,$
\begin{equation}
Q_{n,j}(s,t):=-\tilde{j}\int_{\tilde{j}}^z L_{n-1,\epsilon_j}(s,t,x) e^{2i\tilde{j}\kappa x}dx,\ \  (s,t)\in I^2.
\end{equation}
Noting that $Q_{n,j}$ still possesses certain oscillation, with the
decomposition (\ref{R10E2}) and (\ref{R10E1}) of oscillatory
Volterra integrals, for $1<n\in \bN$ and $(j,\tilde{j},\epsilon_j,
z) \in \aV,$ we define
\begin{equation}
R_{n,j}(s,t):=\tilde{j} e^{2i\kappa}  \tilde{\sigma}_{m}\left[ L_{n-1,\epsilon_j}(s,t,\bcdot)\right](\tilde{j}),\ \ (s,t)\in I^2,
\end{equation}
\begin{equation}\label{R18E51}
G_{n,j}(s,t):=\left(Q_{n,j}(s,t)-R_{n,j}(s,t)\right)e^{-2i\tilde{j} \kappa z} ,\ \ (s,t)\in I^2,
\end{equation}
where $\tilde{\sigma}_{m}$ is defined in the same way as
$\sigma_{m}$ with $\kappa$ being replaced by $2\tilde{j}\kappa$. It
is easy to check that $R_{n,1}=R_{n,2}$ and $R_{n,3}=R_{n,4}$ and
for $(j,\tilde{j},\epsilon_j, z) \in \aV,$
\begin{equation}\label{R18E48}
Q_{n,j}(s,t)=G_{n,j}(s,t)e^{2i\tilde{j} \kappa z}+R_{n,j}(s,t).
\end{equation}
For $1<n\in \bN$ and $\epsilon=1,2,$ we let
\begin{equation}
K_{n,\epsilon}(s,t):= G_{n,\epsilon}(s,t)+J_{n,\epsilon}(s,t) +G_{n,\epsilon+2}(s,t),\ \ (s,t)\in I^2 .
\end{equation}
The functions $K_{n,\epsilon}, \epsilon=1,2$ are the less
oscillatory part of the kernel for the new integral operator which
is separated from the composite operator. To better understand the
functions defined here, readers are referred to the proof of Lemma
\ref{R11L5}. We shall present the property of $K_{n,\epsilon}$,
$n\in \bZ_{m+1}^+$, $\epsilon=1,2$. To this end, we study in the
next lemma the property of $J_{n,\epsilon}$, for $\epsilon=1,2$ and
$1<n\leq m$.

\begin{lemma}\label{R11L9}
Let $1<n\leq m$. If $K\in C^{[m]}(I^2)$ is independent of $\kappa$
and $K_{n-1,\epsilon}, \epsilon=1,2$ are the less oscillatory
kernels of order $n-1$, then $J_{n,\epsilon}$, $\epsilon=1,2$, are
the less oscillatory kernels of order $n$.
\end{lemma}
\begin{proof}
Since $K\in C^{[m]}(I^2)$ is independent of $\kappa$ and $
K_{n-1,\epsilon}\in C^{[m]}_{\kappa,0}(I^2), \epsilon=1,2$, it is
straightforward to show that $J_{n,\epsilon}\in
C^{[m]}_{\kappa,0}(I^2)$, $\epsilon=1,2$. It suffices to prove that
$J_{n,\epsilon}$, $\epsilon=1,2$, satisfy (\ref{R8E36}). To this
end, we consider the maximum bound of $J_{n,\epsilon}^{(l,p)}(s,s)$
for $l+p\leq n-1$. For the concise representation of
$J_{n,\epsilon}^{(l,p)}(s,s)$, let
$S_\alpha(s,t):=L_{n-1,\epsilon}^{(\alpha,0,0)}(s,t,s) ,$
$\alpha\in\bZ_l$ and
$P_\beta(s,t):=L_{n-1,\epsilon}^{(l,\beta,0)}(s,t,t)$,
$\beta\in\bZ_{p}$. A direct calculation with the help of
(\ref{R11E10}) yields
\begin{equation}\label{R18E41}
J_{n,\epsilon}^{(l,p)}(s,s)=(-1)^{\epsilon-1}\left(\sum_{\alpha=0}^{l-1}
S_\alpha^{(l-1-\alpha,p)}(s,s)-\sum_{\beta=0}^{p-1}
P_\beta^{(0,p-1-\beta)}(s,s)\right).
\end{equation}
By using the Leibniz rule, we have that
\begin{equation}\label{R8E3}
S_\alpha^{(l-1-\alpha,p)}(s,s)=\sum_{\gamma=0}^{l-1-\alpha}C_{l-1-\alpha}^\gamma
\left[ K^{(\alpha,0)}(s,s)\right]^{(\gamma)}
K_{n-1,\epsilon}^{(l-1-\alpha-\gamma,p)}(s,s)
\end{equation}
and
\begin{equation}\label{R8E4}
P_\beta^{(0,p-1-\beta)}(s,s)=\sum_{\gamma=0}^{p-1-\beta}C_{p-1-\beta}^\gamma
K^{(l,\gamma)}(s,s) \left. \left[
K_{n-1,\epsilon}^{(0,\beta)}(t,t)\right]^{(p-1-\beta-\gamma)}\right|_{t=s}.
\end{equation}
Since $K\in C^{[m]}(I^2)$ is independent of $\kappa$, the
derivatives of $K$ appeared in  (\ref{R8E3}) and (\ref{R8E4}) are
uniformly bounded by a constant $c$ independent of $\kappa$. Note
that the orders of derivatives of $K_{n-1,\epsilon}$ in (\ref{R8E3})
and (\ref{R8E4}) are at most $n-2$. Recalling that
$K_{n-1,\epsilon}, \epsilon=1,2$ are the less oscillatory kernel of
order $n-1$, the derivative of $K_{n-1,\epsilon}$ in (\ref{R8E3})
and (\ref{R8E4}) is uniformly  bounded by $c\kappa^{l+p+1-n}$. Thus,
$S_\alpha^{(l-1-\alpha,p)}(s,s)$ and $P_\beta^{(0,p-1-\beta)}(s,s)$
are bounded uniformly by $c\kappa^{l+p+1-n}$ for a positive constant
$c$ independent of $\kappa$. By using (\ref{R18E41}) we conclude
that $J_{n,\epsilon}, \epsilon=1,2$, satisfy (\ref{R8E36}).
\end{proof}

Next, we consider functions $R_{n,j}$ and $G_{n,j}$.

\begin{lemma}\label{R11L10}
Let $1<n\leq m$. If $K\in C^{[m]}(I^2)$ is independent of $\kappa$
and $K_{n-1,\epsilon}, \epsilon=1,2$, are the less oscillatory
kernels of order $n-1$, then for $j\in\bZ_5^+$, $R_{n,j}, G_{n,j}\in
C^{[m]}_{\kappa,0}(I^2)$ and $G_{n,j}$ are the less oscillatory
kernels of order $n$.
\end{lemma}
\begin{proof}
Since $K\in C^{[m]}(I^2)$ is independent of $\kappa$ and
$K_{n-1,\epsilon}\in C^{[m]}_{\kappa,0}(I^2)$, $\epsilon=1,2$, by a
proof similar to that of Lemma \ref{R10L2}, $G_{n,j}$ and $R_{n,j}$,
$j\in\bZ^+_5$ can be proved in $C^{[m]}_{\kappa,0}(I^2)$.

We next show that $G_{n,j}$, $j\in\bZ^+_5$ are the less oscillatory kernels of
order $n$. Because $G_{n,j}$, $j\in\bZ^+_5$ have a similar structure, they can be
proved in the same way. Thus, we verify only the case of $G_{n,1}$.
Suppose that $l, p\in\bN$ and $l+p\leq n-1$. From (\ref{R18E51}), $G_{n,1}$ has an explicit expression
\[
G_{n,1}(s,t)=\left[\int_{-1}^tL_{n-1,2}(s,t,x)e^{-2i\kappa x}dx +e^{2i\kappa}\tilde{\sigma}_m\left[L_{n-1,2}(s,t,\cdot)\right](-1)\right] e^{2i\kappa t},\ \ (s,t)\in I^2.
\]
For convenient presentation, we let $S_s(t,x):=L^{(l,0,0)}_{n-1,2}(s,t,x)$.
Then, the $l$-th order partial derivative of $G_{n,1}$ with respect to the variable $s$ has the form
\begin{equation}\label{R18E44}
G_{n,1}^{(l,0)}(s,t)=\left[ \int_{-1}^tS_s(t,x)e^{-2i\kappa x}dx
+e^{2i\kappa} \tilde{\sigma}_{m}\left[S_s(t,\bcdot)\right](-1)
\right]e^{2i\kappa t}.
\end{equation}
Applying Lemma \ref{R11L1} to (\ref{R18E44}) leads to the following explicit expression of $G_{n,1}^{(l,p)}$
\begin{equation}\label{R18E34}
G_{n,1}^{(l,p)}(s,t)=-\sum_{(\alpha,\beta)\in\aU_{m,p}}C_p^\beta
(2i\kappa)^{p-1-\alpha-\beta} S_s^{(\beta,\alpha)}(t,t)
+\sum_{\beta=0}^pC_p^\beta
(2i\kappa)^{p-\beta-m}\int_{-1}^tS^{(\beta,m)}_s(t,x)e^{-2i\kappa(x-t)}dx.
\end{equation}
To show that $G_{n,1}$ is the less oscillatory kernel of
order $n$, we consider the bound of $G^{(l,p)}_{n,1}(s,s)$. We consider the two summations on the right hand side of (\ref{R18E34}) separately.
We first consider the first sum. The terms satisfying $\alpha+\beta\geq n-2$ in the sum are
bounded by $c\kappa^{p+1-n}$ since $p-1-\alpha-\beta\leq p+1-n$. For
any $\alpha+\beta< n-2$, $S_{s}^{(\beta,\alpha)}(s,s)$ is bounded by
$c\kappa^{\alpha+\beta+2-n}$ due to the fact that $K_{n-1,2}$ is the
less oscillatory kernel of order $n-1$. Thus
$\kappa^{p-1-\alpha-\beta}S_s^{(\beta,\alpha)}(s,s)$ is uniformly
bounded by $c\kappa^{p+1-n}$ for all $\alpha,\beta\in \aU_{m,p}$.
Therefore, the first sum on the right hand side of (\ref{R18E34})
is bounded by $c\kappa^{p+1-n}$.

We next consider the second sum. To this end, we let $\aU:=\aU_{m,p}
\cup \{(\alpha, \beta):\alpha=m, \beta\in\bZ_{p+1}\}$ and observe that
\begin{equation}\label{SSS}
S_{s}^{(\beta,\alpha)}(t,x)= \sum_{\gamma=0}^\alpha C_\alpha^\gamma
K^{(l,\gamma)}(s,x)K_{n-1,2}^{(\alpha-\gamma,\beta)}(x,t), \ \
(\alpha,\beta)\in\aU.
\end{equation}
Since $K\in C^{[m]}(I^2)$ is independent of $\kappa$ and $K_{n-1,2}\in  C^{[m]}_{\kappa,0}(I^2)$, there exists a constant $c$ independent of $\kappa$ such that $\|K^{(l,\gamma)}\|_\infty<c$ and $\|K_{n-1,2}^{(\alpha-\gamma, \beta)}\|_\infty<c$  for all $l\leq m$, $\alpha\leq m$, $\beta\leq m$ and $\gamma\leq\alpha$. These estimates together with equation (\ref{SSS}) lead to the estimate $\|S_s^{(\beta,\alpha)}\|_\infty<2^\alpha c^2$.
It follows for $\kappa>1/2$ that
\[
\begin{split}
\left|\sum_{\beta=0}^pC_p^\beta (2i\kappa)^{p-\beta-m}\int_{-1}^tS^{(\beta,m)}_s(t,x)e^{-2i\kappa(x-t)}dx \right| \leq \sum_{\beta=0}^p C_p^\beta(2\kappa)^{p-\beta-m}2^{m+1}c^2 \leq 2^{2p+1}c^2\kappa^{p-m}.
\end{split}
\]
Thus, we conclude that there exists a constant $c_1$ independent of $\kappa$ such that the second term on the right hand side of (\ref{R18E34}) is bounded by $c_1\kappa^{p-m}$.

We now combine the two cases above to obtain the desired result of this lemma.
Noting that $p+1-n$ and $p-m$ are
both no more than $l+p+1-n$, we have that $G^{(l,p)}(s,s)$ is
uniformly bounded by $c\kappa^{l+p+1-n}$ with respect to $s$ for a
positive constant $c$ independent of $\kappa$. Consequently, $G_{n,1}$ is
the less oscillatory kernel of order $n$.
\end{proof}

In the next lemma we establish that $K_{n,\epsilon}$ is the less
oscillatory kernel of order $n$.

\begin{lemma}\label{R11L11}
If $K\in C^{[m]}(I^2)$ is independent of $\kappa$, then
$K_{n,\epsilon}$, $\epsilon=1,2$ are the less oscillatory kernels of
order $n$, $n\in\bZ^+_{m+1}$.
\end{lemma}
\begin{proof}
According to the discussion presented after Lemma \ref{R11L2}, $K$
is the less oscillatory kernel of order $1$. The result of this
lemma may be proved by induction on $n$ with a help of Lemmas
\ref{R11L9} and \ref{R11L10}.
\end{proof}

We need a set of auxiliary integral operators to represent the
decomposition of the IIO. For $n\in\bN$, we define
$\bK_n:=\bA_0[K_{n,1},K_{n,2}]$. In particular, we have that
$\bK_1=\bK$. We present in the next lemma the decomposition for
iterated integral operators of high orders.

\begin{lemma}\label{R11L5}
If $K\in C^{[m]}(I^2)$ is independent of $\kappa$, then for $u\in
H^m_\kappa(I)$
\begin{equation}\label{R18E46}
\bK^n u=\bK_n u+v,\ \ v\in \tilde{H}^m_{\kappa,0}(I), n\in \bZ_{m+1}^+.
\end{equation}
\end{lemma}
\begin{proof}
The proof of this lemma is done by using induction on $n$. The case
for $n=1$ trivially holds. Suppose that the case for $n<m$ is proved
and we consider the case $n+1$. By applying $\bK$ to equation
(\ref{R18E46}), we have that
\[
\bK^{n+1}u=\bK\left(\bK_nu+v\right)=\bK\bK_nu+\bK v.
\]
Since $v\in \tilde{H}^m_{\kappa,0}(I)$, it follows from Corollary
\ref{R11C1} that $\bK v\in \tilde{H}^m_{\kappa,0}(I).$ It remains to
consider the composite operator $\bK\bK_n$. We decompose its kernel
which has the form
\[
\int_{-1}^1K(s,x)K_{n}(x,t)e^{i\kappa(|s-x|+|x-t|)}dx,
\]
where $K_n(x,t):=K_{n,1}(x,t)$ if $x\geq t$ and
$K_n(x,t):=K_{n,2}(x,t)$, otherwise. We next decompose the kernel by
splitting the integral domain to remove the absolute signs. We
consider two cases $t\leq s$ and $t>s$. When $t\leq s$, the integral
domain is split into three parts $[-1, t], [t,s]$ and $[s,1]$. In
this case, the kernel equals
\begin{equation}\label{R18E49}
e^{i\kappa
(s+t)}Q_{n+1,1}(s,t)+e^{i\kappa(s-t)}J_{n+1,1}(s,t)+e^{-i\kappa
(s+t)}Q_{n+1,3}(s,t).
\end{equation}
By substituting (\ref{R18E48}) into (\ref{R18E49}), with the
definition of $K_{n+1,1}$ we write the kernel as
\[
e^{i\kappa (s-t)}K_{n+1,1}(s,t)+e^{i\kappa (s+t)}R_{n+1,1}(s,t)
+e^{-i\kappa (s+t)}R_{n+1,3}(s,t).
\]
In the other case, the integral domain is split into $[-1, s],
[s,t]$ and $[t,1]$ and accordingly the kernel becomes
\[
e^{i\kappa (t-s)}K_{n+1,2}(s,t)+e^{i\kappa (s+t)}R_{n+1,1}(s,t)
+e^{-i\kappa (s+t)}R_{n+1,3}(s,t).
\]
We have already used that $R_{n+1,2}=R_{n+1,1}$ and
$R_{n+1,4}=R_{n+1,3}$. By introducing the function
$$
v_1(s):=\int_IR_{n+1,1}(s,t)e^{i\kappa t}u(t)dte^{i\kappa s}+
\int_IR_{n+1,3}(s,t)e^{-i\kappa t}u(t)dte^{-i\kappa s},
$$
according to the decomposition of the kernel, we have that
\[
\bK\bK_nu=\bK_{n+1}u+v_1.
\]
Since for  $j\in\bZ^+_5$, $R_{n+1,j}\in C^{[m]}_{\kappa,0}(I^2)$, as
proved in Lemma \ref{R11L10} it is straightforward to show that
$v_1\in \tilde{H}^m_{\kappa,0}(I)$. Hence the decomposition
(\ref{R18E46}) for $\bK^{n+1}$ is verified.
\end{proof}

We are now ready to show a property of the IIO of $\bK$, which is
crucial for establishing the main result of this section.

\begin{theorem}\label{R11T3}
If $K\in C^{[m]}(I^2)$ is independent of $\kappa$, then $\bK^n:
H^m_\kappa(I)\rightarrow \tilde{H}^m_{\kappa,m-n}(I)$,
$n\in\bZ_{m+1}^+$.
\end{theorem}
\begin{proof}
Since $K\in C^{[m]}(I^2)$ is independent of $\kappa$, by Lemma
\ref{R11L5}, we have the decomposition that $\bK^n u=\bK_n u+v,\;
v\in \tilde{H}^m_{\kappa,0}(I)$ for $u\in H^m_\kappa(I)$. According
to Lemmas \ref{R11L11} and \ref{R11L2}, we observe that $\bK_n:
H^m_\kappa(I)\rightarrow \tilde{H}^m_{\kappa,m-n}(I)$. Noting that
$\tilde{H}^m_{\kappa,0}(I)\subset \tilde{H}^m_{\kappa,m-n}(I)$, we
conclude that $\bK^n u\in \tilde{H}^m_{\kappa,m-n}(I)$ for $u\in
H^m_\kappa(I)$.
\end{proof}

Theorem \ref{R11T3} implies that $\bK^n$ can reduce the oscillation
order of oscillatory functions in $H^m_\kappa(I)$. In particular,
when $n=m$, we have that $\bK^mH^m_\kappa(I)\subset
\tilde{H}^m_{\kappa,0}(I)$. This is our desired result. As a direct
consequence of Theorem \ref{R11T3}, we obtain the main result of
this section, which concerns the oscillatory property of the
solution of equation (\ref{R9E15}).

\begin{theorem}\label{R11T2}
Let $y$ be the solution of equation (\ref{R9E15}). If $K\in
C^{[m]}(I^2)$ is independent of $\kappa$ and $f\in
\tilde{H}^m_{\kappa,0}(I)$, then
$y\in \tilde{H}^m_{\kappa,0}(I)$.
\end{theorem}
\begin{proof}
According to Theorem \ref{R11P2},
we have that $y\in H^m_\kappa(I)$. Since $K\in C^{[m]}(I^2)$ is
independent of $\kappa$  and $f\in \tilde{H}^m_{\kappa,0}(I)$, by
employing Corollary \ref{R11C1} and Theorem \ref{R11T3}, the result
of this theorem follows naturally from equation (\ref{R18E15}) with
$n=m$.
\end{proof}

Theorem \ref{R11T2} characterizes the oscillatory structure of the
solution and it serves as a base for developing efficient numerical
methods for solving the equation (\ref{R9E15}).

\section{Oscillation preserving Galerkin method}
\label{sec;6}

In this section we shall develop an oscillation preserving Galerkin
method (OPGM) for the numerical solution of equation (\ref{R9E15}).
Specifically, we shall introduce to the usual spline space certain
simple, basic oscillation functions which characterize the
oscillation of the solution of equation (\ref{R9E15}). We shall
establish the convergence theorem for the proposed OPGM.

We begin with describing the usual spline space which is used to
approximate the solution of the equation. Given a partition of $I$,
$\Delta: -1=t_0<t_1<\ldots<t_N<t_{N+1}=1,$ let
$h:=\max_{j\in\bZ_{N+2}^+}(t_{j}-t_{j-1})$ and
$I_j:=(t_{j-1},t_{j})$, $j\in\bZ_{N+2}^+$. We denote by $\Pi_m$ the
set of polynomials of degree less than $m$. Let
$S_h^{m,\nu}(\Delta):=\left\{w\in C^\nu(I):w|_{I_j}\in \Pi_m,
j\in\bZ_{N+2}^+\right\}$ where $-1\leq \nu\leq m-1$ (when $\nu=-1$,
the elements of this space are allowed to possess jump discontinuous
at their  knots). This space is called the space of spline functions
of degree $m-1$ with knots at $t_1, t_2,\ldots,t_N$ of multiplicity
$m-1-\nu$ and its dimension is given by $d:=m(N+1)-N(1+\nu)$. The
smoothest space of nondegenerate splines is the one with $\nu=m-2$
which has dimension $N+m$. Let $S_h^m:=S_h^{m,m-2}(\Delta)$. The
space $S_h^{m}$ has the B-spline basis $\{B_j: j\in\bZ_d\}$ of order
$m$ where $d=N+m$. Let $\bP_h'$ denote the orthogonal projection
from $L^2$ onto $S_h^m$. It is well-known that $\bP_h'$ has the
properties:

 (\romannumeral1)
For all $u\in L^2(I)$, $\|\bP_h'u-u\|_2\rightarrow 0$ as
$h\rightarrow 0$.

(\romannumeral2) There exist constants $c>0$ and $h_0>0$ such that
for all $0<h<h_0$, $\|\bP_h'u-u\|_2\leq ch^m\|u\|_{H^m}$, for all
$u\in H^m(I)$.

\noindent As indicated at the end of Section \ref{sec;2}, approximate
solutions of equation (\ref{R9E15}) from standard spline spaces
suffer from the oscillation of its exact solution. Theorem
\ref{R11T2} suggests that we should include in the usual spline
space additional simple, basic oscillation functions which
characterize the oscillation of the solution of the equation. This
gives the OPGM which we describe next.

According to Theorem \ref{R11T2}, we shall use an extended spline
space $V_{\kappa,h}^m$ based on space $S_h^m$ with the structure
$\{e^{i\kappa s}, e^{-i\kappa s}: s\in I\}$ as our approximate space
for the numerical solutions. Namely, we let
\begin{equation}\label{R11E19}
V_{\kappa,h}^m:=\{w_1e^{i\kappa \bcdot}+w_2e^{-i\kappa \bcdot}+w_3:
w_j\in S_h^m, j\in\bZ_4^+\}
\end{equation}
and call  $V_{\kappa,h}^m$ the structure space. Let $\bP_h$ denote
the orthogonal projection from $L^2$ onto $V_{\kappa,h}^m$. Since
this projection maps a function in $L^2$ into $V_{\kappa,h}^m$ with
an oscillatory structure, we call it an oscillation preserving
projection, and call the solution $y_h$ of the equation
\begin{equation}\label{R9E75}
y_h=\bP_h\bK y_h+\bP_hf,
\end{equation}
or equivalently
\begin{equation}\label{R9E79}
  (y_h,v_h)=(\bK y_h,v_h)+(f,v_h),\ \ \mbox{for all}\ \ v_h\in V_{\kappa,h}^m,
\end{equation}
an OPGM approximate solution in $V_{\kappa,h}^m$ for equation
(\ref{R9E15}). Since the solution $y$ of equation (\ref{R9E15}) and
$V_{\kappa,h}^m$ share the same oscillation structure, $y_h$ can
capture the main oscillation of the solution $y$. We expect that the
numerical solution $y_h$ approximates the $y$ in an optimal order.

In order to analyze $\bP_h$, we define two $\kappa$-oscillatory
spaces
 $W_\tau^m:=\{we^{i\tau \kappa \bcdot}: w\in S_h^m\},$ for $\tau=-1, 1$.
Correspondingly, we define two linear operators $\bP_h^\tau$ from
$L^2$ onto $W_\tau^m$ by
\begin{equation}\label{R11E14}
(\bP_h^\tau u)(s):=(\bP_h'w)(s)e^{\tau i\kappa s}, \;s\in I\; \text{where}\;w(s):=u(s)e^{-\tau i\kappa s},\; \text{for all}\; u\in L^2.
\end{equation}
Note that $((\bP_h^\tau)^2u) (s)=(\bP_h^\tau (\bP_h'w)e^{\tau
i\kappa\bcdot})(s)=(\bP_h'(\bP_h'w))(s)e^{\tau i\kappa s}=
(\bP_h'w)(s)e^{\tau i\kappa s}=(\bP_h^\tau u)(s), s\in I,$ for all
$u\in L^2(I)$. Operators $\bP_h^\tau$ are projection operators from
$L^2$ onto $W_\tau^m$, $\tau=\pm1$.

%

We recall one known result taken from \cite{ANSELONE1971}.
\begin{proposition}\label{R9P1}
Let $X$ be a Banach space. If $\bT, \bT_n: X\rightarrow X$ are
bounded linear operators with $\bT_n\rightarrow \bT$ pointwise as
$n\rightarrow\infty $, then $$\|(\bT_n-\bT)\bL\|\rightarrow 0,
n\rightarrow\infty $$ for each compact operator $\bL: X\rightarrow
X$.
\end{proposition}

To show the error bound of OPGM, we also define for $u\in \tilde{H}^m_{\kappa,0}$,
\[
\|u\|_{\tilde{H}^m_{\kappa,0}}:=\inf\left\{\left(\sum_{j=1}^3 \|w_j\|_{H^m}^2\right)^{1/2}: \  u=w_1e^{i\kappa \cdot}+w_2 e^{-i\kappa\cdot}+w_3, w_j\in H^m_{\kappa,0}, j=1,2,3\right\}.
\]

We are now ready to prove the main theorem of this section which
gives the convergence of OPGM.

\begin{theorem}\label{R9T8}
Suppose that $K\in C^{[m]}(I^2)$ is independent of $\kappa$, $f\in
\tilde{H}^m_{\kappa,0}(I)$ and 1 is not the eigenvalue of $\bK$ for any $\kappa$. If $y$ is the solution of equation
(\ref{R9E15}), then for sufficiently large $\kappa$, there exist
positive constants $c$ independent of $\kappa$ and $h_0>0$ such that
for all $0<h< h_0$, the OPGM has a unique solution $y_h\in
V_{\kappa,h}^m$ that satisfies
\begin{equation}\label{R8E2}
\|y-y_h\|_2\leq ch^m\|y\|_{\tilde{H}^m_{\kappa,0}}.
\end{equation}
\end{theorem}
\begin{proof}
From (\ref{R9E16}) and (\ref{R9E75}), we obtain that
\begin{equation}\label{R18E37}
\left(\bI-\bP_h \bK \right)(y_h-y)=\bP_hy-y.
\end{equation}
Since $\bP_h$ is the orthogonal projection from $L^2$ onto
$V_{\kappa,h}^m$, we have that $\|\bP_h\|=1$. By property (\romannumeral1) of $\bP_h'$ and the inequality
\[
\|\bP_hu-u\|_2\leq (1+\|\bP_h\|)\|\bP_h'u-u\|_2,
\]
we observe for $u\in
L^2(I)$ that $\|\bP_hu-u\|_2\rightarrow 0$, as $h\rightarrow 0$.
Hence, by employing Proposition \ref{R9P1} together with the
compactness of $\bK$, we have that $\|\bP_h\bK-\bK\|\rightarrow 0$,
as $h\rightarrow 0$. The existence of $(\bI-\bK)^{-1}$ ensures that
there exists an $h_0>0$ such that for all $0<h<h_0$,
\begin{equation}\label{R18E17}
\|(\bI-\bK)^{-1}(\bP_h\bK-\bK)\|<\frac{1}{2}.
\end{equation}
Noting that
$$
\bI-\bP_h\bK=(\bI-\bK)\left[\bI+(\bI-\bK)^{-1}(\bK-\bP_h\bK)\right],
$$
we obtain by inequality (\ref{R18E17}) that
\begin{equation}\label{R9E76} \|(\bI-\bP_h\bK)^{-1}\|\leq
\frac{\|(\bI-\bK)^{-1}\|}{1-\|(\bI-\bK)^{-1}(\bP_h\bK-\bK)\|}<2\|(\bI-\bK)^{-1}\|.
\end{equation}
This ensures the existence of $(I-\bP_h\bK)^{-1}$. Thus, we have
from (\ref{R18E37}) that
  \begin{equation}\label{R9E77}
    y-y_h=(\bI-\bP_h\bK)^{-1}(y-\bP_hy).
  \end{equation}

It remains to estimate  $\|y-y_h\|_2$. Theorem \ref{R11T2} ensures
that $y\in \tilde{H}^m_{\kappa,0}(I)$.
 This means that for sufficiently large $\kappa$ there
exist $w_j\in H^m_{\kappa,0}(I), j\in\bZ_4^+$ such that
\begin{equation}\label{R18E14}
y(z)=w_1(s)e^{i\kappa s}+w_2(s)e^{-i\kappa s}+w_3(s),\ \ s\in I.
\end{equation}
Combining (\ref{R9E76}), (\ref{R9E77}) and (\ref{R18E14}) together, we derive that
  \begin{equation}\label{R9E78}
    \|y-y_h\|_2\leq 2\|(\bI-\bK)^{-1}\| (\|w_1e^{i\kappa \bcdot}-\bP_h(w_1e^{i\kappa \bcdot})\|_2+\|w_2e^{-i\kappa \bcdot}-\bP_h(w_2e^{-i\kappa \bcdot})\|_2+\|w_3-\bP_hw_3\|_2).
  \end{equation}
According to Corollary \ref{R11C2}, for sufficiently large $\kappa$ there exists  a positive constant $c$ independent of $\kappa$ such that $\|(\bI-\bK)^{-1}\|\leq c$. It suffices to consider the three terms in the parenthesis in (\ref{R9E78}).
By definition (\ref{R11E14}), we obtain that
  \begin{equation}\label{R18E19}
  \begin{split}
    \|w_je^{i\kappa \bcdot}-\bP_h(w_1e^{i \tau_j\kappa \bcdot})\|_2\leq 2\|w_je^{i\kappa \bcdot}-\bP_h^{\tau_j}(w_je^{i \tau_j\kappa \bcdot})\|_2 = 2\|w_j-\bP_h'w_j\|_2,
  \end{split}
  \end{equation}
where $\tau_j=j-2, j\in \bZ_{4}^+$ and $\bP_h^0:=\bP_h'$. The
desired result (\ref{R8E2}) is obtained by substituting
(\ref{R18E19}) into (\ref{R9E78}) and using property
(\romannumeral2) of $\bP_h'$.
\end{proof}

Noting that for $y\in \tilde{H}^m_{\kappa,0}(I)$,
$\|y\|_{\tilde{H}^m_{\kappa,0}}$ is bounded by a positive constant
independent of $\kappa$, solving the oscillatory equation
(\ref{R9E15}) by OPGM gives the {\it optimal} order of convergence
which equals the order of the spline functions used in the
approximation.

\section{Stability Analysis}
\label{sec;7}

In the section, we investigate the stability of OPGM based on the
B-splines, in terms of the condition number of the coefficient
matrix of the linear system that results from the method. We shall
establish that when $\kappa$ is large enough, the condition number
of the discrete system is bounded.

We first describe the linear system for OPGM based on the B-splines.
Suppose that space $S_h^m$ is defined based on a uniform partition
and the B-spline basis $\{B_j: j\in\bZ_d\}$ of order $m$ is chosen
for space $S_h^m$. By the definition (\ref{R11E19}) of
$V_{\kappa,h}^m$, we have that ${\cal B}:=\{B_j e^{i\kappa\bcdot}$,
$ B_je^{-i\kappa\bcdot},B_j: j\in\bZ_d\}$ is a basis for the
structure space $V_{\kappa,h}^m$. The OPGM approximate solution of
the equation can then be represented by
$$
y_h(z)=\sum_{j=1}^da_{1,j}B_j(s)e^{i\kappa
s}+\sum_{j=1}^da_{2,j}B_j(s)e^{-i\kappa
s}+\sum_{j=1}^da_{3,j}B_j(s), \ \ s\in I,
$$
where coefficient vector $\Ba:=[a_{l,j}: l\in\bZ_4^+,j\in\bZ_d]^T$,
is determined by the linear system obtained by substituting the
above expression of $y_h$ in (\ref{R9E79}) and replacing $v_h$ by
each basis function in ${\cal B}$. To set up the linear system, we
let $\epsilon_p:=p-1$, for $p\in\bZ_3$, and for $p\in \bZ_3$, we
define the vectors $\Bf_{p}:= \left[\left(f, B_j
e^{i\epsilon_p\kappa \bcdot}\right): j\in \bZ_d\right]^T$ and use
them as blocks to form the vector  $\Bf:=\left[\Bf_{p}:
p\in\bZ_3\right]$. For $p,q\in \bZ_3$, we introduce matrices
\[
\BE_{p,q}:=\left[\left(B_l e^{i\epsilon_p\kappa \bcdot}, B_j
e^{i\epsilon_q\kappa \bcdot}\right): l,j\in\bZ_d\right],\;\BK_{p,q}
:=\left[\left(\bK\left(B_l e^{i\epsilon_p\kappa \bcdot}\right),
B_je^{i\epsilon_q\kappa \bcdot}\right): l,j\in\bZ_d\right]
\]
and define $\BE_d:=\left[\BE_{p,q}: p,q\in\bZ_3\right]$ and
$\BK_d:=\left[\BK_{p,q}: p,q\in\bZ_3\right].$ Thus, the discrete
linear system (\ref{R9E79}) of OPGM can be rewritten in the matrix
form
$$
(\BE_d-\BK_d)\Ba=\Bf_d.
$$

We next study the condition number of matrix $(\BE_d-\BK_d)$. For
this purpose, we recall two well-known facts.

The first fact concerns the condition number of the coefficient
matrix of a general Galerkin method, which was presented in
\cite{ATKINSON1997} (pp. 94-97). Suppose that the operator equation
$(\lambda-\bK)y=f$ is solved by the Galerkin method with the basis
$\{\phi_0,\ldots,\phi_{d-1}\}$ using the orthogonal projection
operator $\bP_n$. The resulting approximate operator equation is
given by $(\lambda-\bP_n\bK)y=\bP_nf$. We use $\BA_n$ to denote its
coefficient matrix. Let $\Gamma_n:=[(\phi_l,\phi_j): l,j\in\bZ_d]$
denote the Gram matrix of the basis. Then the condition number of
the coefficient matrix  $\BA_n$ satisfies the bound
\begin{equation}\label{R18E16}
{\rm Cond}(\BA_n) \leq {\rm Cond}(\Gamma_n) {\rm
Cond}(\lambda-\bP_n\bK),
\end{equation}
where Cond$(\BA)$ denotes the condition number of the matrix $\BA$.

The second fact regards the condition number of the B-spline basis
of order $k$ \cite{BOOR1973, BOOR19762}. We denote by $\Bt:=\{t_j:
j\in\bZ_{d+k}\}$ the equal-space sequence of knots with the first
$k$ and last $k$ knots in $\Bt$ being the same and $t_0=-1$,
$t_{d+k-1}=1$. Let ${\cal N}:=\{N_j: j\in \bZ_d\}$ denote the
B-spline basis of order $k$ on $\Bt$, that is
\begin{equation}\label{R8E5}
N_j(t):=(t_{j+k}-t_j)[t_{j},\ldots,t_{j+k}](\bcdot-t)_+^{k-1},\ \ j\in\bZ_d, t\in I,
\end{equation}
where $[t_j,\ldots,t_{j+k}]u$ denotes the $k$--th order divided
difference of $u$ at the nodes $t_j,\ldots,t_{j+k}$ and
$\alpha_+:=\max\{\alpha,0\}$. The $L^2$ condition number of the
basis $\{N_j\}$ is defined as
\[
{\rm Cond}_{\it k}:=\sup_{\Bt}\left\{\sup_{\Bb\neq 0}\frac{\|\Bb\|_{l^2}}{\|\sum_{j\in\bZ_d} b_jN_j\|_2}
\sup_{\Bb\neq0}\frac{\|\sum_{j\in \bZ_d} b_jN_j\|_2}{\|\Bb\|_{l^2}}\right\}.
\]
It is proved that there exists a positive constant $D_k$ which only depends on $k$ such that,
\begin{equation}\label{R18E8}
 {\rm Cond}_k<D_k.
\end{equation}
We denote by $\BG:=[(N_l,N_j): l,j\in \bZ_d]$ the Gram matrix of the
B-spline basis. The following lemma shows that the condition number
of $\BG$ is bounded by a constant.

\begin{lemma}\label{R11L13}
If ${\cal N}$ is the B-spline basis of order $k$ defined by
(\ref{R8E5}), then Cond$(\BG)$ is bounded by a constant independent
of the order of the matrix.
\end{lemma}
\begin{proof}
We shall connect Cond $(\BG)$ with  ${\rm Cond}_k$ and make use of
estimate (\ref{R18E8}). Since matrix $\BG$ is symmetric, we have
that
\begin{equation}\label{R8E6}
  {\rm Cond}(\BG)=\lambda_{\max}(\BG)/\lambda_{\min}(\BG).
\end{equation}
We next associate ${\rm Cond}(\BG)$ with the condition number of the
basis ${\cal N}$ through the maximum and minimum eigenvalues. For
this purpose, we derive an alternative expression for the ${\rm
Cond}_k$. Recall that $\left\|\sum_{j\in\bZ_d}
b_jN_j\right\|_2^2=\Bb^T\BG \Bb$. It is straightforward to have that
\[
\sup_{\Bb\neq 0}\frac{\Bb^T\BG \Bb}{\|\Bb\|_{l^2}^2}=\sup_{\Bb\neq 0} \sum_{j\in\bZ_d} \lambda_j \frac{c_j^2}{\sum_{l\in\bZ_d}c_l^2} =\lambda_{\max}(\BG)
\]
and
\[
\sup_{\Bb\neq 0}\frac{\|\Bb\|_{l^2}^2}{\Bb^T\BG \Bb}=1/\inf_{\Bb\neq 0}\frac{\Bb^T\BG \Bb}{\|\Bb\|_{l^2}^2}=1/\inf_{\Bb\neq 0} \sum_{j\in\bZ_d} \lambda_j \frac{c_j^2}{\sum_{l\in\bZ_d}c_l^2} =1/\lambda_{\min}(\BG).
\]
These formulas together with the definition of ${\rm Cond}_k$, we
get that
\begin{equation}\label{R18E9}
{\rm Cond}_k=\sup_{\Bt}\sqrt{\lambda_{\max}(\BG)/\lambda_{\min}(\BG)}.
\end{equation}
It follows from (\ref{R8E6}) and (\ref{R18E9}) that
\begin{equation}\label{R8E9}
{\rm Cond}(\BG)\leq {\rm Cond}_k^2.
\end{equation}
We thus obtain the desired result from (\ref{R18E8}) and
(\ref{R8E9}).
\end{proof}

We now present the theorem regarding the condition number of the
matrix $\BE_d-\BK_d$.

\begin{theorem}\label{R11T4}
Let $K\in C^{[m]}(I^2)$ be independent of $\kappa$, $f\in
\tilde{H}^m_{\kappa,0}(I)$ and $y$ be the solution of equation
(\ref{R9E15}). If the equation (\ref{R9E15}) is solved numerically
by OPGM based on the B-spline basis with the uniform partition, then
there exist constants $c$ and $h_0$ such that for all $0<h<h_0$
there exists a constant $\kappa_0$ such that for all
$\kappa\geq\kappa_0$, Cond$(\BE_d-\BK_d)\leq c$.
\end{theorem}
\begin{proof}
According to the bound (\ref{R18E16}) of the condition number of
discrete system of general Galerkin methods, we have that
$$
{\rm Cond}(\BE_d-\BK_d) \leq {\rm Cond}(\BE_{\it d}){\rm
Cond}(\bI-\bP_h\bK).
$$
By Corollary \ref{R11C2}, there exist
constants $c_1$ and $\kappa_1$ such that for all
$\kappa\geq\kappa_1$, ${\rm Cond}(\bI-\bK)< c_1$. Since
$\bI-\bP_h\bK\rightarrow \bI-\bK$, as $h\rightarrow 0$, there exists
a constant  $h_0$  such that for all $0<h<h_0$ and $\kappa\geq
\kappa_1$, ${\rm Cond}(\bI-\bP_h\bK)<c_1$.

Next, we bound the condition number of $\BE_d$ for all $h<h_0$. We
write $\BE_{0,0}$ as $\BE$ and observe that $\BE_{p,p}= \BE$, for
$p\in\bZ_3.$ When $p,q\in\bZ_3$, $p\neq q$, by the Riemann-Lebesgue
Lemma, the element in $\BE_{p,q}$ tends to 0 as $\kappa\rightarrow
\infty$ and thus, we get that
\begin{equation}\label{R18E4}
 \lim_{\kappa\rightarrow\infty}\BE_d={\rm Diag}(\BE,\BE,\BE).
\end{equation}
We notice that $\BE$ is the Gram matrix of the B-spline basis of
order $m$ which is independent of $\kappa$. Lemma \ref{R11L13}
ensures that there exists a positive constant $c_2$ independent of
$h$ and $\kappa$ such that
\begin{equation}\label{R18E6}
{\rm Cond}(\BE)< c_2.
\end{equation}
Due to the fact that ${\rm Diag}(\BE,\BE,\BE)$ and $\BE$ are
positive-definite matrices, we obtain that
\begin{equation}\label{R18E5}
{\rm Cond}\left({\rm
Diag}(\BE,\BE,\BE)\right)=\frac{\lambda_{\max}\left({\rm
Diag}(\BE,\BE,\BE)\right)} { \lambda_{\min}\left({\rm
Diag}(\BE,\BE,\BE)\right)  }= \frac{\lambda_{\max}(\BE)}
{\lambda_{\min}(\BE)}={\rm Cond}(\BE).
\end{equation}
Combining (\ref{R18E4}), (\ref{R18E6}) and (\ref{R18E5}), we have
that there exists a constant $\kappa_h$ such that for all
$\kappa>\kappa_h$, ${\rm Cond}(\BE_d)\leq c_2.$ We note here that
the value of $\kappa_h$ depends on $h$. By setting $c:=c_1 c_2 $ and
$\kappa_0:=\max\{\kappa_1,\kappa_h\}$, for all $\kappa>\kappa_0$,
Cond$(\BE_d-\BK_d)$ is bounded by $c$.
\end{proof}

Theorem \ref{R11T4} ensures that OPGM based on the B-spline basis is
numerically stable when the wavenumber is large enough.

\section{Numerical Experiments}
\label{sec;8}

In this section, we present a numerical example to demonstrate the
approximation accuracy, the order convergence and the stability of
the proposed method. We also compare it with the conventional
Galerkin method (CGM).

In this example, for simplicity we choose $K(s,t)=1$. For the
numerical comparison purpose, we choose $f$ as
\[
\begin{split}f(s)=&\left(-\frac{1}{4}s^4+(1-\frac{i}{2\kappa})s^3+\frac{3}{4\kappa^2}s^2
+\frac{3i}{4\kappa^3}s-\frac{3}{8\kappa^4}+\frac{i}{k}e^{i\kappa } +\frac{1}{4}\right) e^{i\kappa s} \\
&\;-\left[\frac{e^{2i\kappa }}{8\kappa^4}(-3+6i\kappa+6\kappa^2-4i\kappa^3)-\frac{i}{\kappa}e^{i\kappa}\right] e^{-i\kappa s}+1-\frac{2i}{\kappa}, \ \ s\in I
\end{split}
\]
so that the exact solution of equation (\ref{R9E15}) is given by
$y(s):=1+s^3e^{i\kappa s}$, $s\in I.$ It is straightforward to
obtain the $L^2$ norm of $y$ which is $\|y\|_2=\frac{4\sqrt{7}}{7}$
for any $\kappa$.

We choose the uniform partition $\Delta:=\{t_j:=-1+jh:
j\in\bZ_{N+2}\}$ for interval $I$, where $h:=\frac{2}{N+1}$, with
$t_{-1}=t_0$ and $t_{N+2}=t_{N+1}$. We choose $m=2$ and the B-spline
basis $\{B_{j,2}: j\in\bZ_{N+2}\}$ of order 2 as a basis for
$S_h^2$. That is, $B_{j,2}$ is a piecewise linear function whose
support is $[t_{j-1}, t_{j+1}]$ and $B_{j,2}(t_j)=1$, for
$j\in\bZ_{N+2}$. The CGM will use the approximation space
$S_h^2$ while the OPGM will use the non-$\kappa$-oscillatory
structured space $V_{\kappa,h}^2$ based on $S_h^2$. For a fair
comparison, we shall choose different $N$ values for the two
methods. We use $\tilde{N}$ to represent the order of the
coefficient matrices. We have that $\tilde{N}=3(N+2)$ for the OPGM
and $\tilde{N}=N+2$ for the CGM. To make
the dimension of discrete matrices comparable, the values of $N$ for
the CGM are set four times of those of $N$
for the OPGM. From Theorem \ref{R9T8}, we have that
$\|y-y_h\|_2=\bO(h^2)$ where $y$ is  the exact solution of equation
(\ref{R9E15}) and $y_h$ is the OPGM approximation in
$V_{\kappa,h}^2$. The relative error in the $L^2$ norm is evaluated
numerically by the trapezoidal quadrature rule with 2048 uniformly
distributed notes, which is
\begin{equation}\label{R11E8}
\frac{\|y-y_{h}\|_2}{\|y\|_2}\approx
e_N:=\frac{\sqrt{7}}{4}\left(\sum_{j=1}^{2048}\frac{1}{2048}|y(s_j)-y_{h}(s_j)|^2
\right)^{1/2},
\end{equation}
where  $s_j=-1+\frac{j}{1024}$, for $j\in\bZ_{2049}^+$. We use C.O.
to denote the computed convergence order, which is defined by
\[
\text{C.O.}:=\log\left(\frac{e_N}{e_{2N}}\right)/\log(2).
\]
Likewise, we let $\bar{e}_N$ denote the relative error of the
approximate solution generated by the CGM
and its computed convergence order is defined accordingly.

\begin{table}
\caption{\footnotesize The related errors of the CGM (Left) and the OPGM (Right) with different values of
$\kappa$} \label{R9Ta6} \centering \small \subtable[\footnotesize
$\kappa=50$ ]{
\begin{tabular}{ccc||ccc}
$\tilde{N}$ &$\bar{e}_N$&C.O.& $\tilde{N}$ &$e_N$&C.O.\\
\hline
$66$&$4.23e-2$&$  $&$     54$&$1.36e-3$&$  $\\
$130$&$1.08e-2$&$1.98 $&$ 102$&$7.24e-4$&$0.92 $\\
$258$&$2.75e-3$&$1.97 $&$198$&$1.46e-4$&$2.31 $\\
$514$&$7.03e-4$&$1.97 $&$390$&$1.25e-5$&$3.54 $\\
$1026$&$1.93e-4$&$1.87 $&$774$&$1.20e-6$&$3.38 $\\
$2050$&$5.20e-5$&$1.89 $&$1542$&$1.17e-6$&$0.04 $\\
\end{tabular}
\label{R9Ta1}
}
\subtable[\footnotesize  $\kappa=500$]{
\begin{tabular}{ccc||ccc}
$\tilde{N}$&$\bar{e}_N$&C.O.&$\tilde{N}$&$e_N$&C.O.\\
\hline
$66$&$2.50e-1$&$ $&$54$&$1.67e-3$&$ $\\
$130$&$2.49e-1$&$0.00 $&$102$&$4.54e-4$&$1.88 $\\
$258$&$2.38e-1$&$0.06 $&$198$&$1.20e-4$&$1.92 $\\
$514$&$7.60e-2$&$1.65 $&$390$&$2.67e-5$&$2.16 $\\
$1026$&$1.93e-2$&$1.98 $&$774$&$1.98e-5$&$0.43 $\\
$2050$&$5.18e-3$&$1.90 $&$1542$&$3.06e-6$&$2.69 $\\
\end{tabular}
\label{R9Ta2}
}
\subtable[\footnotesize  $\kappa=5000$ ]{
\begin{tabular}{ccc||ccc}
$\tilde{N}$ &$\bar{e}_N$&C.O.&$\tilde{N}$ &$e_N$&C.O.\\
\hline
$66$&$2.50e-1$&$ $&$54$&$1.67e-3$&$ $\\
$130$&$2.50e-1$&$-0.00 $&$102$&$4.55e-4$&$1.87 $\\
$258$&$2.50e-1$&$-0.00 $&$198$&$1.19e-4$&$1.93 $\\
$514$&$2.50e-1$&$-0.00 $&$390$&$3.04e-5$&$1.97 $\\
$1026$&$2.41e-1$&$0.05 $&$774$&$7.69e-6$&$1.98 $\\
$2050$&$2.30e-1$&$0.07 $&$1542$&$1.93e-6$&$1.99 $\\
\end{tabular}
\label{R9Ta3}
}
\subtable[\footnotesize $\kappa=50000$ ]{
\begin{tabular}{ccc||ccc}
$\tilde{N}$ &$\bar{e}_N$&C.O.&$\tilde{N}$ &$e_N$&C.O.\\
\hline
$66$&$2.50e-1$&$  $&$54$&$1.67e-3$&$ $\\
$130$&$2.50e-1$&$-0.00 $&$102$&$4.55e-4$&$1.87 $\\
$258$&$2.50e-1$&$-0.00 $&$198$&$1.19e-4$&$1.93 $\\
$514$&$2.50e-1$&$-0.00 $&$390$&$3.04e-5$&$1.97 $\\
$1026$&$2.50e-1$&$-0.00 $&$774$&$7.70e-6$&$1.98 $\\
$2050$&$2.50e-1$&$0.00 $&$1542$&$1.94e-6$&$1.99 $\\
\end{tabular}
\label{R9Ta4}
}
\end{table}

\begin{figure}[!ht]
  \centering
  \includegraphics[width=4in,height=2.8in]{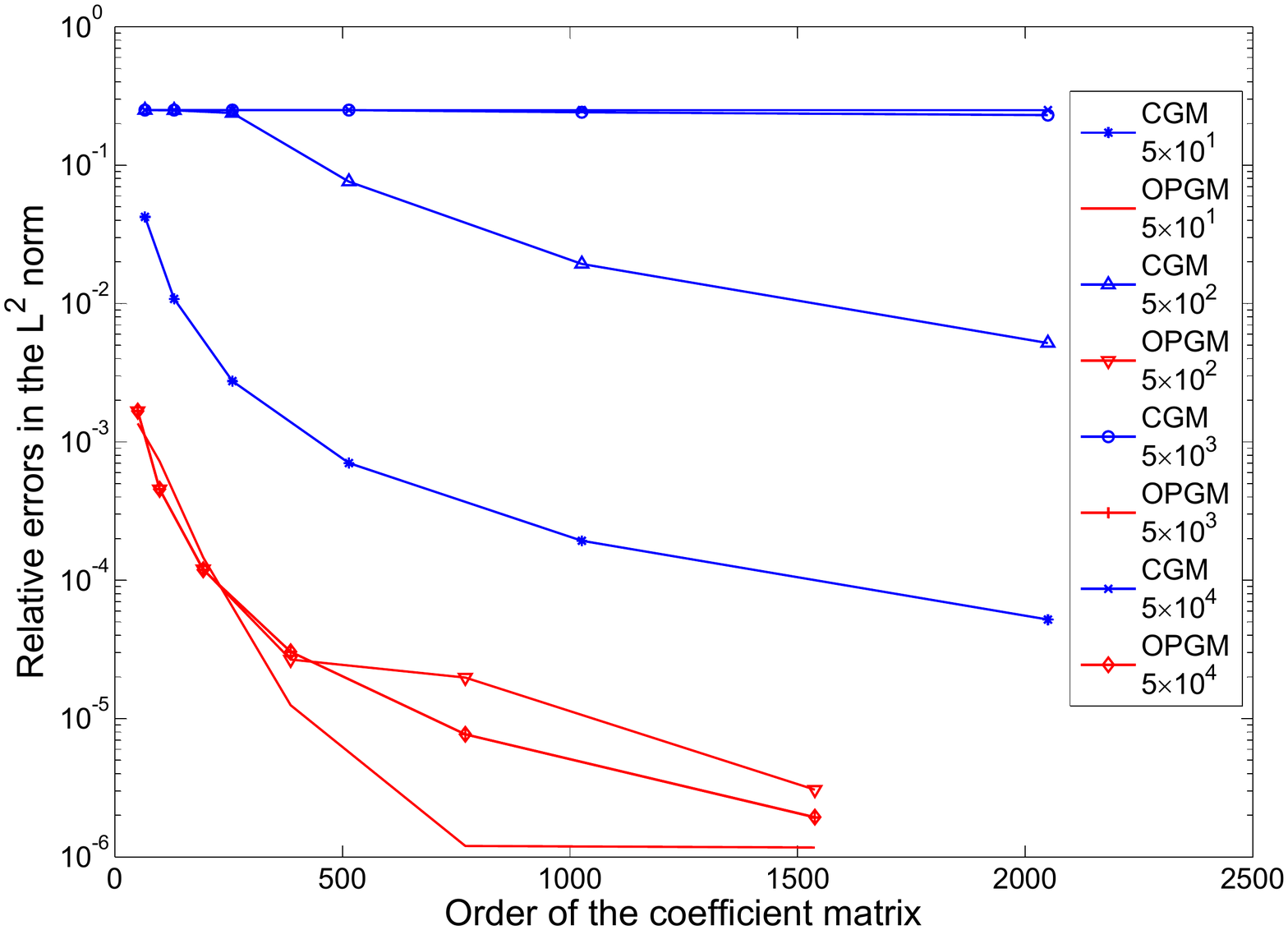}
  \caption{The relative errors in the $L^2$ norm versus the order of the matrix.
   The numbers in the legend are the values of the wavenumber.}\label{R18F1}
\end{figure}
\begin{figure}[!ht]
  \centering
  \includegraphics[width=4in,height=2.5in]{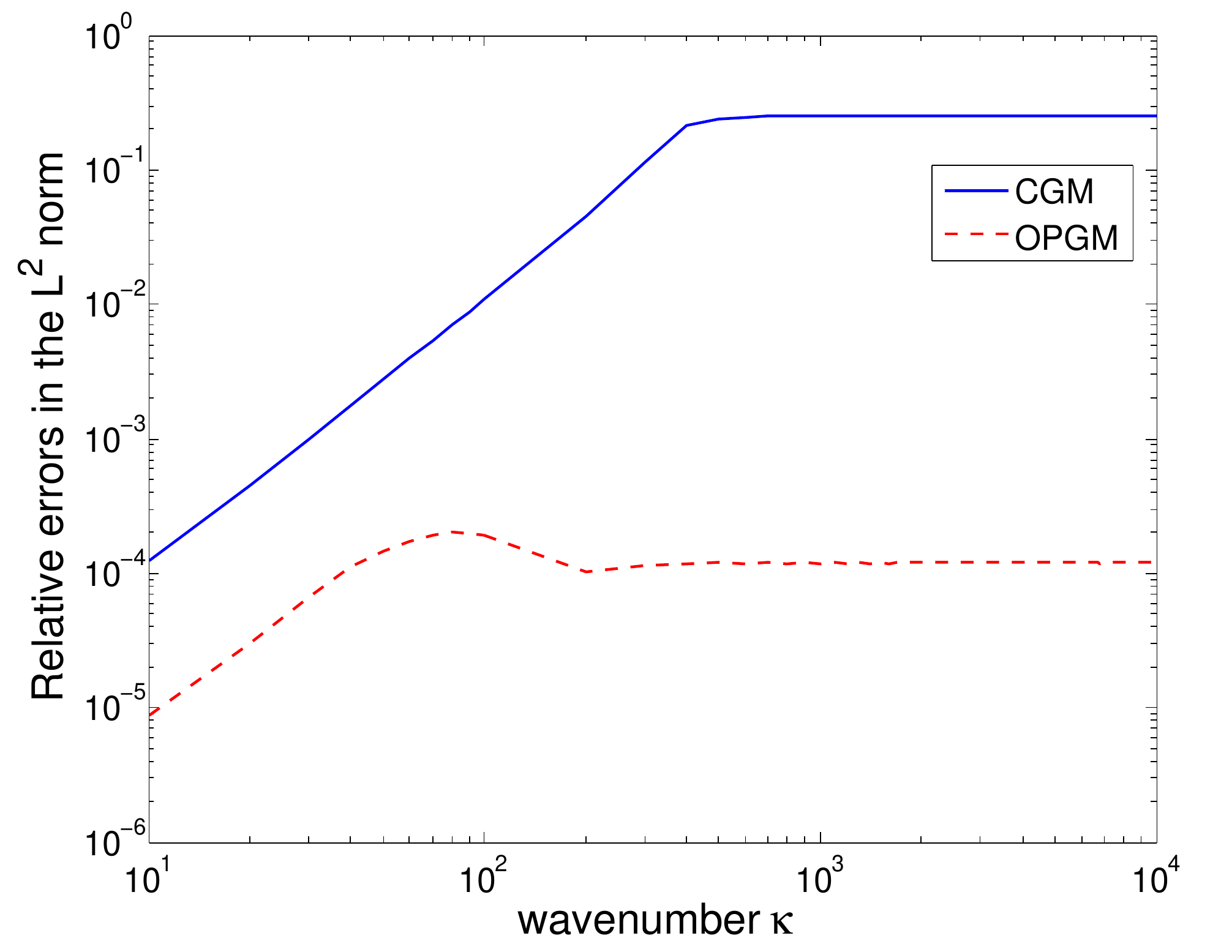}
  \caption{The relative errors in the $L^2$ norm  versus the wavenumber.
   The orders of the  matrices are 258 and 198 for the CGM and the OPGM, respectively.}
\label{R18F3}
\end{figure}
\begin{figure}[!ht]
  \centering
  \includegraphics[width=4in,height=2.5in]{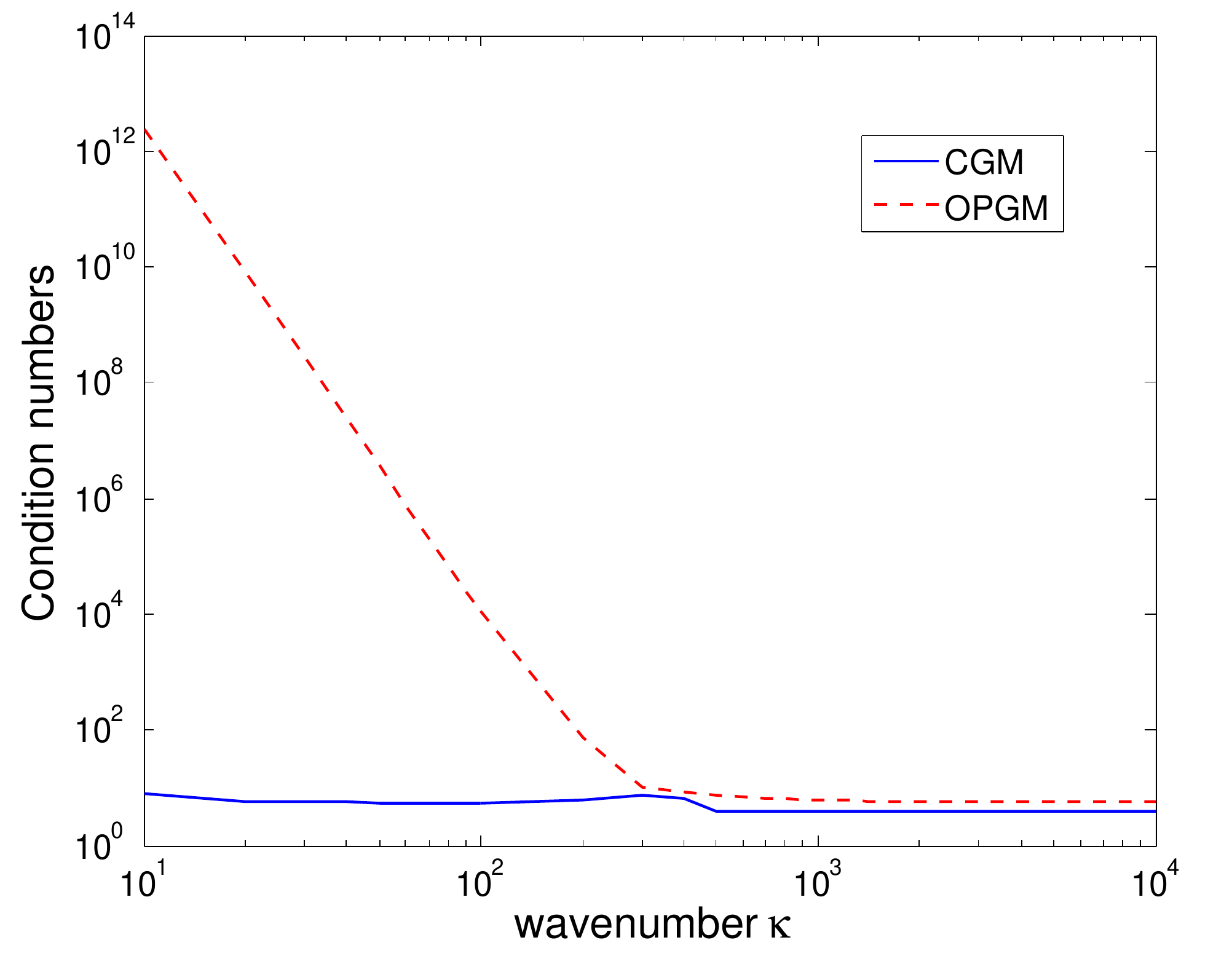}
  \caption{Condition numbers versus the wavenumber.
  The orders of the matrices are 258 and 198 for the CGM and the OPGM, respectively.}\label{R18F4}
\end{figure}

We perform two experiments. In experiment one, we solve equation
(\ref{R9E15}) with fixing the wavenumber $\kappa$ to $50, 500, 5000$
and $50000$ and varying the order $\tilde{N}$ of the coefficient
matrices. We report the numerical results of this experiment in
Table \ref{R9Ta6}. They are also plotted in Figure \ref{R18F1}.
These numerical results show that the OPGM are much more accurate
than the CGM for all chosen $\kappa$
values. It is also illustrated in Fig. \ref{R18F1}. In particular,
when $\kappa=5000$ or $50000$, the CGM
totally fails to solve the equation while the OPGM still preserves
the approximation accuracy. The convergence order of the OPGM is
shown to be $\bO(h^2)$ in Table \ref{R9Ta6} (c) and (d) when
$\kappa$ is large enough which verifies the result of Theorem
\ref{R9T8}.

In experiment two, we solve equation (\ref{R9E15}) with fixing the
order of the coefficient matrices and varying the wavenumber from 10
to $10^4$. The orders of the coefficient matrices for the
CGM and the OPGM are set to be 258 and 198,
respectively. We plot in Fig \ref{R18F3} the errors of both the
methods with respect to the wavenumber and in Fig \ref{R18F4} the
condition number of the coefficient matrix with respect to the
wavenumber.  From Fig \ref{R18F3}, we have two observations: The
CGM is greatly affected by the wavenumber
before the method fails to get an approximation with an acceptable
accuracy; the OPGM nearly keep the uniform error with respect to the
wavenumber when $\kappa$ is more than about 200. It verifies that
the convergence of the OPGM is independent of the wavenumber. In Fig
\ref{R18F4}, when $\kappa$ is larger than $300$, the condition
number of the coefficient matrix of the OPGM is uniformly bounded.
This is consistent with that of the CGM.
This confirms the theoretical result in Theorem \ref{R11T4}. Though
the CGM has the uniform condition number
bound independent of $\kappa$, it cannot give an appropriate
approximation when $\kappa$ becomes large. This is because the
standard spline spaces cannot approximate well an oscillatory
function. In summary, the OPGM
gives the optimal convergence order and at the same time preserves
the uniform condition number of the CGM.

To close this section, we note that to solve equation (\ref{R9E15}),
we need to evaluate three kinds of integrals to obtain the linear
system: (\romannumeral1) $(B_{l,2}e^{\tau_p i\kappa s}$,
$B_{j,2}e^{\tau_q i\kappa s}),$ $p\in\bZ_4^+$, $q\in\bZ_4^+$;
(\romannumeral2) $(f, B_{j,2}e^{\tau_p i\kappa s})$, $p\in\bZ_4^+$;
(\romannumeral3) $(\bK B_{l,2}e^{\tau_p i\kappa s}$,
$B_{j,2}e^{\tau_q i\kappa s})$, $p\in\bZ_4^+$, $q\in\bZ_4^+$, where
$B_{j,2}$ is a B-spline basis function and $\tau_p:=p-2$,
$p\in\bZ_4^+$. Integrals of class (\romannumeral1) can be evaluated
explicitly. The other two classes may be computed numerically. In
this paper, since $K(s,t)=1, (s,t)\in I^2$ and the expression of $f$
is simple, we can calculate the elements of the linear system
through explicit expressions with the help of (\ref{R8E37}).
However, for general $K$ and $f$, we have to compute oscillatory
integrals of classes (\romannumeral2) and (\romannumeral3)
numerically. Highly oscillatory integrals have been understood
deeply and calculated extremely efficiently by many methods, such as
asymptotic method, Filon-type method \cite{NORSETT2005,XIANG2007},
Levin-type method \cite{OLVER2006}, steepest descent method
\cite{HUYBRECHES2006} and Clenshaw-Curtis-Filon-type method
\cite{DOMINGUEZ2013,DOMINGUEZ2011,XIANG2011}. A new method is
proposed in \cite{MA2014} by using the idea of graded mesh to
analyze oscillatory integrals which is proved to be even more
efficient than most of the existing methods.

\section*{Acknowledgment}

This work was partially supported by the National Science Foundation
of USA under grants DMS-1115523, by the National Natural Science
Foundation of China under grants 11271370, 61171018, 11071286,
91130009 and 11101439, by Guangdong Provincial Government of China
through the Computational Science Innovative Research Team program,
the Doctor Program Foundation of Ministry of Education of China
under grant 20100171120038, and by the CSC Scholarship. The first
author would like to thank sincerely Prof. Jianshu Luo for his
generous support.

\bibliographystyle{plain}

\end{CJK}
\end{document}